\documentclass[11pt]{amsart}
\usepackage[utf8]{inputenc}
\usepackage{a4wide}
\usepackage[T1]{fontenc}
\usepackage{amssymb}
\usepackage[english]{babel}
\usepackage{a4wide}
\usepackage{xcolor}
\usepackage{enumitem}
\usepackage[pagebackref, colorlinks = true, linkcolor = teal, urlcolor  = teal, citecolor = violet]{hyperref}
\usepackage{dsfont}
\usepackage{ulem}

\newcommand{\xRightarrow}[2][]{\ext@arrow 0359\Rightarrowfill@{#1}{#2}}
\newtheorem{thm}{Theorem}[section]
\newtheorem{prop}[thm]{Proposition}

\theoremstyle{plain}
\newtheorem{lem}[thm]{Lemma}
\theoremstyle{plain}
\theoremstyle{definition}

\theoremstyle{remark}

\newtheorem{rmq}{Remark}
\newcommand{\Var}{\text{VaR}_{\alpha}(X)}
\newcommand{\CVar}{\text{CVaR}_{\alpha}(X)}

\newcommand{\tEn}{\tilde{\mathcal{E}}_n}
\newcommand{\tEnun}{\tilde{\mathcal{E}}_n^1}
\newcommand{\tEndeux}{\tilde{\mathcal{E}}_n^2}
\newcommand{\Tznp}{\tilde{\mathbb{Z}}_{n+1}}
\newcommand{\Tzn}{\tilde{\mathbb{Z}}_n}
\newcommand{\un}{\mathds{1}}
\newcommand{\Znp}{Z_{n+1}}
\newcommand{\Zn}{Z_n}
\newcommand{\Tztn}{\tilde{\mathbb{Z}}^{(n)}}
\newcommand{\tZnp}{\widetilde{Z}_{n+1}}
\newcommand{\tZn}{\widetilde{Z}_{n}}
\newcommand{\tZnpd}{\widetilde{Z}_{n+1}^{(2)}}
\newcommand{\tZnd}{\widetilde{Z}_{n}^{(2)}}
\newcommand{\ta}{\theta_{\alpha}}
\newcommand{\tn}{\theta_{n}}
\newcommand{\btn}{\bar{\theta}_{n}}
\newcommand{\tTn}{\tilde{\mathbb{T}}_n}
\newcommand{\tTnp}{\tilde{\mathbb{T}}_{n+1}}
\newcommand{\Tn}{\mathbb{T}_n}

\newcommand{\Vtn}{\mathbb{V}_{n}}
\newcommand{\Vtnp}{\mathbb{V}_{n+1}}
\newcommand{\con}{\mathcal{O}_n}
\newcommand{\tcon}{\widetilde{\mathcal{O}}_n}
\newcommand{\tconp}{\widetilde{\mathcal{O}}_{n+1}}
\newcommand{\tconn}[1]{\tcon^{(#1)}}
\newcommand{\conp}{\mathcal{O}_{n+1}}
\newcommand{\btnp}{\bar{\theta}_{n+1}}
\newcommand{\kan}{k_{\alpha,n}}

\newcommand{\tnp}{\theta_{n+1}}
\newcommand{\vta}{\vartheta_{\alpha}}
\newcommand{\vtn}{\vartheta_{n}}
\newcommand{\vtnp}{\vartheta_{n+1}}
\newcommand{\bvtn}{\widehat{\vartheta}_{n}}  % modifié
\newcommand{\bvtnp}{\widehat{\vartheta}_{n+1}} %modifié
\newcommand{\Fn}{\mathcal{F}_{n}}

\newcommand{\E}{\mathbb{E}}

\newcommand{\R}{\mathbb{R}}

\newcommand{\N}{\mathbb{N}}

\newcommand{\M}{\mathcal{M}}
\newcommand{\MM}{\mathbb{M}}

\newcommand{\F}{\mathcal{F}}

\renewcommand{\and}{\mbox{ and }}
\usepackage{xargs}
\usepackage[prependcaption]{todonotes}

\title[Non asymptotic controls for a superquantile approximation algorithm]{Non asymptotic controls on a recursive superquantile approximation}
\author{ M. Costa $\ddagger$}
\address{$\ddagger$ Institut de Math\'ematiques de Toulouse, UMR 5219\\Université de Toulouse, CNRS\\ UPS, F-31062 Toulouse Cedex 9, France}
\email{manon.costa@math.univ-toulouse.fr}
\author{ S. Gadat $\S$  }
\address{$\S$   Toulouse School of Economics, Institut Universitaire de France}
\email{sebastien.gadat@tse-fr.eu}
\date{\today}

\begin{document}
\thanks{S. G. gratefully acknowledges financial support from the Agence Nationale de la Recherche (MaSDOL grant ANR-19-CE23-0017).  S. G. is member of
the Institut Universitaire de France (IUF), and part of this work has been carried out with
financial support from the IUF. M. C. has been supported by the Chaire “Mod\'elisation Math\'ematique et Biodiversit\'e” of Veolia
Environnement-\'Ecole Polytechnique-Museum national d'Histoire naturelle-Fondation X. P. C.
%M. C. and S. G. would like to thank Bernard Bercu for fruitful discussions about their work.
}
	
\begin{abstract}
In this work, we study a new recursive stochastic algorithm for the joint estimation of quantile and superquantile of an unknown distribution. The novelty of this algorithm is to use the Cesaro averaging of the quantile estimation \textit{inside} the recursive approximation of the superquantile. We provide some sharp non-asymptotic bounds on the quadratic risk of the superquantile estimator for different step size sequences. We also prove new non-asymptotic $L^p$-controls on the Robbins Monro algorithm for quantile estimation and its averaged version. Finally, we derive a central limit theorem of our joint procedure using the diffusion approximation point of view hidden behind our stochastic algorithm.

\end{abstract}

\subjclass[2010]{Primary: 62L20; Secondary: 60F05; 62P05}
\keywords{Stochastic approximation; Quantile and Superquantile; Non-asymptotic controls, Diffusion approximation.}
\maketitle

%\tableofcontents
\section{Introduction}
\subsection{Quantiles and superquantiles}
In this paper, we are interested in the estimation of a standard financial risk measure with a recursive stochastic algorithm. Let be given a real random variable $X$ that mimics the outcome of a portfolio of some financial assets, \cite{axiomatique} introduces a set of axioms that shall describe in actuarial science and financial economics a \textit{coherent} risk measure (denoted by $\rho$ in this introduction). One of these key properties is that $\rho$ must bring the diversification principle: the risk of two portfolios together is less than the two individual risks, which means from a mathematical point of view that for two random variables $Z_1$ and $Z_2$, $\rho$ satisfies $\rho(Z_1+Z_2) \leq \rho(Z_1)+\rho(Z_2)$. In some recent works, this sub-additive property has been relaxed and replaced by a convex axiom: $\rho(\lambda Z_1+(1-\lambda) Z_2) \leq \lambda \rho(Z_1)+(1-\lambda)\rho(Z_2)$.
If the value at risk (VaR), \textit{i.e.} a quantile of a random variable, is a very popular measure of risk in finance, it appears that it is not a coherent risk measure since it does not satisfy the \textit{diversification principle}. Oppositely, \cite{Artzner97} introduced a Conditional Value at Risk (CVaR) measure (or expected shortfall), which is shown to be a coherent risk measure in \cite{Pflug00}.
More precisely, if $X$ refers to the outcome of a portfolio, we define by $F$ the cumulative distribution function:
$$
\forall x \in \R \qquad F(x) = \mathbb{P}(X \leq x),
$$
and for a given $\alpha \in (0,1)$, the quantile $\theta_{\alpha}$ is denoted by:
\begin{equation}\label{def:quantile}
\ta  := \inf  \left\{\theta  \in \R  \, \vert \, F(\theta) \geq \alpha \right\}.
\end{equation}
Note that in many works in actuarial and financial mathematics, the notation $\Var$ or $ES_{\alpha}(X)$ (for expected shortfall) is used instead of $\ta$. For a given $\alpha \in (0,1)$, the superquantile is then defined as soon as the random variable $X$ is $L^1$, with the help of:
\begin{equation}\label{def:super_quantile}
\vta := \mathbb{E}[ X \, \vert \, X \geq \ta] = \frac{\mathbb{E}[X \un_{X \geq \ta}]}{1-\alpha}.
\end{equation}
Superquantiles are often adopted as a measure of risk when modeling a risk-averse decision maker and are commonly denoted by $\CVar$ in financial economics and mathematics. More precisely,
$\vta$  quantifies a risk-averse constraint on asset return prices: for a (log)-return weekly price $Z$ of a portfolio, we are commonly interested in the estimation or in some guarantees of the super-quantile of $-Z$ when $\alpha = 95\%$, which translates the average loss induced by the worst $5\%$ scenario.
Besides 
being a coherent risk measure, the superquantile also translates more information on the tail of the distribution of the random variable $X$.\\

In our paper, we develop a new method for jointly estimating $(\ta,\vta)$ and obtain some non-asymptotic guarantees for the quadratic loss of estimation, which is a novel type of result for such a kind of estimation problem with stochastic algorithms. We also obtain a central limit theorem associated to our method.

\subsection*{Related works}
Estimating quantiles has a longstanding history in statistics: except in parametric models where explicit formulas are available, the estimation of quantiles is a real issue. It may be tackled with the help of Monte-Carlo approaches when some simulation tools of the random variable $X$ are available, or using the Extreme Value Theory when $\alpha$ is close to $0$ or $1$ (see \textit{e.g.} \cite{EVT,Embrechts99}). 
Another popular point of view is to approximate the behaviour of the $\Var$ with correcting linear and quadratic terms (see \textit{e.g.} \cite{NonlinearVar1,NonlinearVar2,NonlinearVar3}).
Finally, a large amount of work is based on order statistics to derive some estimation procedures of $\Var$. We refer to \cite{Garivier1} for an overview of some recent uses of order statistics for quantile estimation with some prospects in the field of computer experiments.

The development of new mathematical studies on superquantiles started with \cite{Ben-Tal} and \cite{Cvar}, and use some variational formulation of $\CVar$ as a solution of a convex optimization problem, to derive some estimation strategy. In particular, their approach does not require any parametric form of the distribution of $X$. In \cite{Garivier2}, the authors exploit some Monte-Carlo strategy coupled with some order-statistics to estimate superquantiles and they establish a central limit theorem associated with their estimation method. 
Oppositely, many (econometric or finance) works are developed with an important parametric a priori modeling (see \textit{e.g.} \cite{Hall-Yao,Wang-Zhao,Embrechts}) and owing to their parametric or semi-parametric point of view, the methods cited above obviously suffer from robustness issues.

All the previous methods are based on a batch framework, where the user observes the whole set of samples before starting the estimation procedure. Nevertheless, in some concrete situations, the observations may only record on-line, while in other big-data situations, the observations are simply too numerous to be handled with a batch method. In these cases, we  are led to consider some recursive strategies where we estimate quantiles and superquantiles using the observations sequentially, with the help of a stochastic approximation algorithm (see, \textit{e.g.}, \cite{Kushner_Yin03,benveniste_metivier_priouret}).
The recursive quantile algorithm is a classical example of such a method (see \textit{e.g.} \cite{Duflo97}) and has led to recent numerous contributions to the generalization of geometric medians (see, among others  \cite{godichon2015,Cenac,Godichon}). 

The classical algorithm can be written as:
\begin{equation}\label{eq:algo_classique}
\left\{ \begin{aligned}
  &\tnp =\tn-a_n \left( \un_{X_{n+1} \leq \tn}-\alpha\right)\\
  &\vtnp = \,\vtn+b_n \left( \frac{X_{n+1}}{1-\alpha}\un_{X_{n+1} >\tn} - \vtn\right) 
\end{aligned}\right.
\end{equation}
Finally, the recent work of \cite{bardou2009computing} developed a different stochastic algorithm that estimates both $\CVar$ and $\Var$ within the same joint procedure. Indeed, \cite{bardou2009computing} used the current value of $(\tn)_{n \geq 1}$ and modified $(\vtn)_{n \ge 1}$ into:
$$\begin{aligned}
\widetilde{\vartheta}_{n+1}&=\widetilde{\vartheta}_n+b_n \left( \tn+\frac{(X_{n+1}-\tn)}{(1-\alpha)}\un_{X_{n+1} >\tn} - \widetilde{\vartheta}_n\right)\\
&=\widetilde{\vartheta}_n+b_n(L(\tn,X_{n+1})-\widetilde{\vartheta}_n),
\end{aligned}
$$
to estimate $\vta$. Their main idea is to consider a convexified version of the algorithm since the update function $L(\theta,x)=\theta+\left(\frac{x-\theta}{1-\alpha}\right)\un_{x>\theta}$  satisfies that $\theta\mapsto\E(L(\theta,X))$ is convex. Then they use a Ruppert-Polyak averaging procedure, introduced in the seminal contributions \cite{Ruppert,polyakjuditsky}, to obtain a central limit theorem for their stochastic algorithm. In the recent work \cite{BCG1}, we developed a purely asymptotic study of these two algorithms and derive almost sure properties (Quadratic Strong Law and Law of the Iterated Logarithm) as well as a central limit theorem.
%We give in Section \ref{subsec:comp} more details on the link between our work and existing results. 
 
%\seb{Modif ici}

\subsection{Algorithm}
 
In this paper, we develop a stochastic procedure, that may be closely related to the one of \cite{bardou2009computing,BCG1}, to jointly estimate the quantile and the superquantile of an unknown distribution: we expect to benefit from the acceleration of the Cesaro averaging procedure with $(\btn)_{n \geq 1}$ to obtain a better recursion on the superquantile estimation.\\
For this purpose, we first introduce the following three-dimensional stochastic algorithm, which is a generalization of the stochastic algorithm \eqref{eq:algo_classique}:
\begin{equation}\label{eq:algo}
\left \{
\begin{aligned}
    &\tnp =\tn-a_n \left( \un_{X_{n+1} \leq \tn}-\alpha\right)\\
    &\btnp = \btn \frac{n}{n+1} + \frac{1}{n+1} \tnp = \displaystyle{\frac{1}{n} \sum_{k=1}^{n+1} \theta_k}\\
    &\bvtnp = \,\bvtn+b_n \left( \frac{X_{n+1}}{1-\alpha}\un_{X_{n+1} >\btn} - \bvtn\right) 
\end{aligned}
\right. .
\end{equation}

The sequence $(\tn)_{n \geq 1}$ corresponds to the standard recursive quantile algorithm (see \textit{e.g.} \cite{Duflo97}), whereas $(\btn)_{n \geq 1}$ is the Cesaro average over the past iterations of the initial stochastic gradient descent $(\tn)_{n\geq1 }$. Cesaro averaging is known as an efficient way for reducing the quadratic risk of estimation of $\ta$ (see \textit{e.g.} \cite{Ruppert,polyakjuditsky}). We will provide a sharp non-asymptotic analysis of $(\tn)_{ n \geq 1}$ and $(\btn)_{n \geq 1}$ below since it will be necessary for our study of $(\bvtn)_{n \geq 1}$. Our approach is similar to \cite{Gadat-Panloup} even though for our purpose, we need to derive some upper bounds of the $L^p$ risk of $(\btn)_{n \geq 1}$ for $p = 4$ instead of $p=2$ in \cite{Gadat-Panloup}.
We compare the algorithm proposed in this article with existing ones at the end of this section.

%We obtain the consistency of $(\bvtn)_{n \geq 1}$ for the estimation of $\vta$ using an i.i.d. sequence of $(X_n)_{n \geq 1}$ distributed according to $\mathbb{P}$.\\

We do not assume any parametric form of this underlying distribution, we only assume that this distribution is absolutely continuous with respect to the real one-dimensional Lebesgue measure $\lambda$, whose density is denoted by $f$:
$$
\forall x \in \R \qquad 
f(x) := \frac{\text{d}\mathbb{P}}{\text{d}\lambda}(x).
$$
We consider the general situation where $f$ is supported over $\mathbb{R}$, \textit{i.e.} we do not make any assumption on the compactness of the support of $f$ and \textit{we do not assume that $f$ is lower bounded by any non-negative constant}.
%\seb{citer tous les travaux qui ne supposent pas cela.}

\medskip 

\noindent
\textbf{Assumption $\mathbf{H}_f$:}
We assume that $f(\ta)>0$, that $\|f\|_{\infty} < + \infty$ and that  $\theta \longmapsto (1+|\theta|) |f'(\theta)|$ is a bounded function. 
We further assume that $X$ has a moment of order strictly larger than $2$:
 $$ \exists p > 2: \qquad \int x^p f(x) \text{d}x < \infty.$$
\noindent
In particular,  $\mathbf{H}_f$ implies that $f$ is L-Lipschitz for a large enough L.
This assumption is very close to the framework used in \cite{bardou2009computing}. The boundedness of $\theta \longmapsto (1+|\theta|) |f'(\theta)|$ is crucial to obtain the non asymptotic bounds and some precises controls of the first and second order terms.
\medskip 

\noindent
\textbf{Assumption $\mathbf{H}_{(a_n,b_n)}$}
The gain sequences satisfy:
$$
a_n =a_1 n^{-a} \qquad \text{and} \qquad b_n = b_1 (n+1)^{-b} \quad \text{with} \quad \frac{1}{2} < a<b \le 1.
$$
This corresponds to the standard  framework of Robbins-Monro algorithm. We specifically consider the case of two-scale algorithms, similarly to \cite{MokkademPelletier2006}. We emphasize that the case where $a=b$ was studied in \cite{bardou2009computing} and the case $b<a$ in \cite{BCG1}.  Here, we also consider the complementary situation where the learning rate of the super-quantile algorithm is faster than the one of the quantile sequence.
\medskip

\subsection{Main results}
We derive non asymptotic estimates of the quadratic loss as well as a $L^p$ upper bound of the recursive quantile estimation $(\btn)_{n\geq 1}$, with $p \geq 1$. We introduce the $L^{p}$ risk associated to $\btn-\ta$ denoted by:
$\MM_{n,2p}$:
$$
\MM_{n,2p} = \E |\btn-\ta|^{2p}.
$$
\begin{thm}\label{theo:rate_averaging}
Assume $\mathbf{H}_f$  and $\mathbf{H}_{(a_n,b_n)}$, then:\begin{itemize}
\item[$i)$] A constant $\Gamma_a$ exists such that the recursive quantile algorithm satisfies:
$$
\E (\btn-\ta)^2 
%\leq \frac{\alpha(1-\alpha)}{f(\ta)^2 n} + \Gamma_a% n^{-\left( \frac{3}{2}-\frac{a}{2}\right)}.
%\manon{n^{-\left(\frac{1}{2}+a\right) \wedge \left( 2-a\right)},}
\leq \frac{\alpha(1-\alpha)}{f(\ta)^2 n} + \kappa' \frac{\alpha(1-\alpha)}{f(\ta)^3} n^{-\left( 2-a\right)\wedge \left(\frac{1}{2}+a\right)} + \Gamma_a n^{-\left( 3-a\right)\wedge \left(\frac{3}{2}+a\right)}\quad \forall n\ge1,
$$
where $\kappa'$ is an explicit constant given in \eqref{def:kappa'}.\\
% $$\kappa'=\frac{2\max(L K_2 a_1, 2(a+1)/a_1, 2a_1)}{\left( 3-a\right)\wedge \left(\frac{3}{2}+a\right)}.$$
The optimal choice of $a$ corresponds to $a=3/4$ and in this case
$$
\E (\btn-\ta)^2 \leq \frac{\alpha(1-\alpha)}{f(\ta)^2 n} +  \frac{8\kappa \alpha(1-\alpha)}{7f(\ta)^3} n^{-5/4} + \Gamma_{3/4} n^{-9/4},\quad \forall n\ge1.
$$
\item[$ii)$] For any integer $p \geq 1$, a constant $c_p$ exists such that:
$$
\forall n \geq 1 \qquad
\MM_{n,2p} \leq c_p n^{-p}.
$$ 
\end{itemize}
\end{thm}

The first item $i)$ of Theorem \ref{theo:rate_averaging} is a consequence of Corollary 6 in \cite{Gadat-Panloup}. However, we precise the computations of the second order term in the framework of the quantile estimation and we obtain precise constants for the second order term. We also point out that this second order term involves $f(\ta)^{-3}$. We observe that in our case, if the second order term involved the local curvature given by the Cramer-Rao lower bound $f(\ta)^{-2}$, this curvature is also involved in the second order term with a worse power ($3$ instead of $2$), that is of course compensated by $n^{-5/4}$.  As indicated in \cite{Gadat-Panloup}, the second order term is better regarding  its dependency with $n$ but maybe degraded by the curvature or the dimension of the ambient space.\\
Even though the second result $ii)$ is stated for a general value of $p$, we emphasize that this result is weaker than that of $i)$ when $p=2$: the first item provides a first order \textit{optimal} result, while the constant $c_p$ given by the second item is pessimistic. 
The first order term is optimal thanks to the Ruppert-Polyak central limit theorem satisfied by $(\btn-\ta)\sqrt{n}$ (see \textit{e.g.} \cite{polyakjuditsky,Ruppert}).
Nevertheless, $ii)$ of Theorem \ref{theo:rate_averaging} will be essential to derive a satisfactory linearization of our superquantile algorithm $(\bvtn)_{n \geq 1}$. Therefore, we should consider Theorem \ref{theo:rate_averaging} as a result by itself and as a cornerstone tool for the next result as well. \\
We also state in Theorem \ref{theo:momentp} similar results for the Robbins Monro algorithm $(\tn)_{n\ge1}$, however in this setting we have no explicit expression of the first order constant.
\medskip 
 
\noindent
Below, we study the properties of the superquantile recursive estimation algorithm and we differentiate our results into two cases that depend on the choice of the sequence $(b_n)_{n\ge 0}$. We first consider the case $b_n=b_1/(n+1)$ and then $b_n=b_1(n+1)^{-b}$ with $1/2<b<1$.\\
We introduce the important notation $V_{\alpha}$:
\begin{equation}
\label{def:Valpha}
V_{\alpha} :=\mathbb{V}\left(X\mathds{1}_{X>\ta}\right)= \int_{\ta}^{+\infty} x^2 f(x) dx -  \left( \int_{\ta}^{+ \infty} x f(x) dx\right)^2.
\end{equation}

\begin{thm}\label{theo:rate_vtn}
Assume $\mathbf{H}_f$  and $\mathbf{H}_{(a_n,b_n)}$.
\begin{itemize}
\item[$i)$] If $b\in(1/2,1)$, then a large enough constant $\Gamma$ exists  such that 
$$
\forall n \ge 1\,, \qquad \E((\bvtn-\vta)^2) \leq \frac{V_\alpha}{2(1-\alpha)^2}b_n + \Gamma n^{-\frac{b+1}{2}}.
$$
%The second order term is optimal when $b>a\ge2/3$.
\item[$ii)$] If $b=1$ and $b_1>(\frac{1+a}{2})\wedge (\frac{5}{2}-a)$, then a large enough constant $\Gamma$ exists such that:
$$\forall n \ge 1\,, \qquad \E((\bvtn-\vta)^2) \le \frac{C_{\alpha, b_1}}{n}+\Gamma n^{-(1+\frac{a}{2})\wedge (2-a)},$$
where
$$C_{\alpha,b_1}= \frac{4b_1^2\alpha(1-\alpha)}{(2b_1-1)^2f(\ta)^2}\left[1+\sqrt{1+\frac{V_\alpha f(\ta)^2(2b_1-1)}{4\alpha(1-\alpha)^3}} \right]^2.
$$

\end{itemize}
\end{thm}

Theorem \ref{theo:rate_vtn} provides some statistical guarantees on the behaviour of $(\bvtn)_{n \geq 1}$ in a finite horizon $n$ while our assumptions on $f$ are very weak: we only assume the smoothness of $f$ with $f(\ta)>0$. The second assumption seems necessary otherwise the distribution $\mathbb{P}$ is flat around $\ta$ and the definition of the quantile and of the super-quantile may be ambiguous.

We also state a sharp asymptotic analysis of the sequence $(\bvtn)_{n \geq 1}$ and we derive a central limit theorem with an explicit computation of the limiting variance.

\begin{thm}\label{theo:tcl}
Assume $\mathbf{H}_f$  and $\mathbf{H}_{(a_n,b_n)}$ .
\begin{itemize}
\item[$i)$] If $b \in (1/2,1)$,
 then:
$$
 \sqrt{b_n^{-1}}\left(\bvtn-\vta \right) \underset{n \to +\infty}{\overset{\mathcal{L}}{\Longrightarrow}} \mathcal{N}\left(0, \frac{V_\alpha }{2(1-\alpha)^2}\right).
$$
\item[$ii)$] If $b=1$,
and $b_1 > 1/2$ then:
$$
\sqrt{n} \begin{pmatrix}
\btn-\ta\\\bvtn-\vta 
\end{pmatrix}\underset{n \to +\infty}{\overset{\mathcal{L}}{\Longrightarrow}} \mathcal{N}(0,S^2)$$
where:
$$
S^2 = \begin{pmatrix}
\frac{\alpha(1-\alpha)}{f(\ta)^2} & \frac{\alpha}{f(\ta)}(\vta-\ta) \\ \frac{\alpha}{f(\ta)}(\vta-\ta) & \frac{b_1^2}{(2b_1-1)} \frac{V_\alpha}{(1-\alpha)^2}-\frac{2b_1}{2b_1-1} \frac{\alpha \ta(\vta-\ta)}{(1-\alpha)}
\end{pmatrix}.
$$
\end{itemize}
\end{thm}

\subsection*{Discussion}
We complete below our previous discussion regarding the results we obtained in Theorem \ref{theo:rate_vtn} and Theorem \ref{theo:tcl}.

\noindent
\textit{Low computational cost.}

\noindent
It is well known that on-line methods lead to very simple and fast ways to compute estimators comparing for example to estimators based on order statistics or plug-in. Hence, it is somewhat classical but however important, to remind how fast our algorithm is to manage the estimation of $\vta$ recursively. Moreover, our method may be adapted to a stream of observations, contrary to batch methods.

\noindent
\textit{Non-asymptotic result.}

\noindent
Gathering $i)$ of  Theorem \ref{theo:rate_vtn} and $i)$  of Theorem \ref{theo:tcl}, we observe that our upper bound $i)$ cannot be improved as it involves the lowest possible constant.
Regarding now $ii)$ of  Theorem \ref{theo:rate_vtn} and $ii)$ of Theorem \ref{theo:tcl},  we observe that the rate of convergence in Theorem \ref{theo:rate_vtn} $ii)$ is of the order $1/n$ and is optimal. Hence, we obtain with our sequential method a parametric rate of estimation of $\vta$. %We will indeed prove that our procedure is estimating $(\ta,\vta)$ with this optimal parametric rate.
Unfortunately,  we emphasize that $C_{\alpha,b_1}$ given in $ii)$ of Theorem \ref{theo:rate_vtn}  does not match with the limiting variance obtained in Theorem \ref{theo:tcl} so that the constant involved in $ii)$ of Theorem \ref{theo:rate_vtn} may certainly be improved.
Comparing the results $i)$ and $ii)$ of Theorems \ref{theo:rate_vtn} and \ref{theo:tcl}, we finally observe that from an asymptotic point of view, the choice $b_n = b_1/(n+1)$ always outperform the other one.

To the best of our knowledge, estimating the super-quantile $\vta$ with a recursive method has only been  addressed by  \cite{bardou2009computing} with the help of a Cesaro averaging procedure on the initial sequence $(\vtn)_{n \geq 1}$ and in \cite{BCG1} where some asymptotic results are established for another method that produces a sequence different from $(\bvtn)_{n \geq 1}$. The results in \cite{bardou2009computing} and in \cite{BCG1} are somewhat weaker than that of Theorem \ref{theo:rate_vtn} since with the same set of assumptions, they only derive asymptotic results  instead of a non-asymptotic bound on the mean square error. 
Our contribution stated in Theorem \ref{theo:rate_vtn} is significantly stronger: we obtain a non-asymptotic upper bound of the performance of our estimator, which possesses the good convergence rate $O(b_n)$ with a sequential strategy.

\noindent
\textit{Limiting variance and comparison with existing results}

\noindent
Regarding now the asymptotic properties, there exists roughly speaking three axes of work in the literature. The first family of work is concerned by parametric or semi-parametric approaches: some parametric models are considered for the sequence of observations and then a plug-in parametric estimation is used with either a direct formula (if available for $\vta$) or a Monte-Carlo integral approximation. It is the point of view adopted in \cite{Chernozhukov}, \cite{Hall-Yao} or \cite{Wang-Zhao} for example. Although efficient from a numerical point of view, these approaches suffer from robustness problems essentially due to some questionable parametrization. 
Consequently, even if some of these previous works also establish central limit theorem with explicit limiting variance of estimation, they cannot be directly compared to our result stated in Theorem \ref{theo:tcl}.

Oppositely, nonparametric (or historical) $\Var$ or $\CVar$ estimation is a robust alternative over the (semi-)parametric approaches and among these works, we can mention batch methods and on-line ones that can manage a stream of observations.

$\bullet$
Concerning the batch approaches, they  are commonly based either on order statistics, kernel smoothing or on the variational description of $\CVar$. 
For example, in \cite{Garivier2}, the authors derive a central limit theorem for the estimation of the super-quantile using the standard approach with order statistics and an assumption on the behaviour of the quantile function that entails an assumption on the tail of the distribution of the random variable $X$, which yields a non-adaptive method regarding the tail parameter of the distribution.
In \cite{Beutner}, the authors use a plug-in method to derive a central limit theorem, with the help of an ad hoc functional delta method. In the same way, \cite{Chernozhukov} also uses a variational approach of $\CVar$ to derive an asymptotic functional central limit theorem for the estimation of $\vta$ whose limiting variance depends on the tail index used in their model.
In all these works, the variances are difficult to compare with the one stated in Theorem \ref{theo:tcl}, at least from a theoretical point of view, while we emphasize that these methods are computationally longer than our on-line method. We also point out that in our work, we do not need any very restrictive assumption on the tail of the distribution of $X$.

$\bullet$ Therefore, it seems that the meaningful methods we have to compare with are those of \cite{BCG1}, \cite{bardou2009computing} and \cite{Garivier2}.
In \cite{Garivier2}, the superquantile is estimated using a Monte-Carlo method and some order statistics of the distribution. The authors obtain a central limit theorem (Theorem 3.1) and the limiting variance whose expression is difficult to compare with our because of the lack of explicit computations. More interestingly, \cite{BCG1} states a central limit theorem for the algorithms $(\vtn)$ and $(\widetilde{\vartheta}_n)$ (see Theorem 3.4). We obtain that both $\sqrt{n^{b}}(\vtn-\vta)$ and $\sqrt{n^{b}}(\widetilde{\vartheta}_n-\vta)$ converge to a Gaussian random variable with variance 
\begin{equation*}
\Gamma_{\vta}= \left \{
\begin{array}[c]{ccc}
{\displaystyle \frac{b_1^2 \tau^2_{\alpha}}{2b_1 - 1}}  & \text{if} & b=1, \vspace{1ex} \\
{\displaystyle \frac{b_1 \tau^2_{\alpha}}{2}} & \text{if} & b<1,
\end{array}
\right. \quad\text{where}  \quad
\tau_\alpha^2=\frac{V_\alpha}{(1-\alpha)^2} - \Bigl(\frac{\alpha\ta}{1-\alpha}\Bigr)(2\vta-\ta).
\end{equation*} 
In the case where $b<1$ and $\vta$ and $\ta$ are positive, the asymptotic variance obtained in Theorem \ref{theo:tcl} $ii)$ is larger than $\Gamma_{\vta}$.\\
The discussion in the case where $b=1$ is more interesting. For a fixed value of $b_1$ we observe that $$S^2_{22}< \Gamma_{\vta} \quad  \Longleftrightarrow \quad b_1<\frac{2\vta-2\ta}{2\vta-\ta} = 1-\frac{\ta}{2\vta-\ta}.$$
Therefore if the distribution $f$ satisfies that $\frac{\vta}{\ta}>\frac{3}{2}$, there exists $b_1$ such that the asymptotic variance of $\bvtn$ is smaller that the one of $\vtn$ and $\widetilde{\vartheta}_n$.\\
Now for a given distribution, the question remains how to find a value of $b_1$ that minimizes $S^2_{22}$ and $\Gamma_{\vta}$.
The minimal $\Gamma_{\vta}$ is equal to $\tau^2_\alpha$ and is reached for $b_1=1$. This value corresponds to the limiting variance of the central limit theorem for the averaged version of $\widetilde{\vartheta}_n$ stated in \cite{bardou2009computing} (Theorem 2.4). The variance $S^2_{22}$ admits a minimum for some $b_1\in(1/2,1)$. However the comparison of this minimal value with $\tau_\alpha^2$ is difficult to handle for a generic distribution. We nevertheless point out that our new algorithm $(\bvtn)_{n \ge 1}$ shall improve the limiting variance of estimation  in comparison with the one of $(\vtn)_{n \ge 1}$ or $(\widetilde{\vartheta}_n)_{n \ge 1}$ when $\vta$ is significantly larger than $\ta$, which is a common feature with heavy-tail distributions like Student, Weibull or Pareto ones. This last remark may be in particular useful for possible applications in finance and actuary.  In these both fields, CVaR estimation is at the cornerstone for designing optimal policies either for contracting or edging. 
We emphasize that in finance, it is commonly admitted that asset return prices may possess heavy tails (see \textit{e.g.} \cite{Mandelbrot,Rachev}) whereas actuarial scientists generally face optimal reinsurance problems with fat tails (see \textit{e.g.} \cite{Balbas,Delbaen}).

\subsection*{Organisation of the paper and notations}
As indicated above, the main goal of this paper is to prove Theorem \ref{theo:rate_vtn} and Theorem \ref{theo:tcl}. To do so, we will also establish some  results on the stochastic gradient descent $(\tn)_{n \geq 1}$ involved in the quantile estimation, and on the averaged sequence $(\btn)_{n \geq 1}$ for the quantile estimation itself. Despite their own interests, these results are necessary to study $(\bvtn)_{n \geq 1}$ and to the best of our knowledge, are new for the recursive quantile algorithm $(\tn)_{n \geq 1}$ and in particular give a state-of-the art result for the sequence $(\bvtn)_{n \geq 1}$ when compared to known results stated in \cite{Duflo97,Cenac,godichon2015}  (among others).  

The paper is organized as follows.
Section \ref{sec:quantile} presents a careful study of the $L^p$ loss of the recursive quantile. The results derive from a strategy based on a family of Lyapunov functions $V_q$ (see Equation \ref{def:Vq}). In particular, we establish Theorem \ref{theo:momentp} that states that $\mathbb{E}[(\tn-\ta)^{2p}]$ is less than $a_n^p$ up to a constant that depends on $p$. In Section \ref{sec:averaging_quantile}, we prove an optimal non-asymptotic result for $(\btn)_{n \geq 1}$. Our proof relies on a spectral analysis of the algorithm, close to the strategy developed in \cite{Gadat-Panloup}.\\
Section \ref{sec:superquantile} then uses Theorem \ref{theo:rate_averaging} to derive the main result of the paper on the super-quantile estimation proof of Theorem \ref{theo:rate_vtn} and  Section \ref{sec:tcl} describes the central limit theorem with the help of an approximation of the rescaled stochastic algorithm by Markov diffusion processes.

\medskip

Below, we will use $a.s.$ to refer to an almost sure convergence result. 
For two positive sequences $(t_n)_{n \geq 1}$ and $(s_n)_{n \geq 1}$, $t_n \lesssim s_n$ refers to an inequality true for any value of $n$ up to a constant $C$ independent of $n$: $t_n \leq C s_n$. 
\\In the sequel we will always denote $(\Var,\CVar)$ by $(\ta,\vta)$.

\section{Quantile estimation}\label{sec:quantile}
This paragraph deals with the statistical estimation of $\ta$ with the sequence $(\tn)_{n \geq 1}$ and the Cesaro averaging procedure defined by
$
(\btn)_{n \geq 1}.$
\subsection{Convergence of the quantile algorithm $(\tn)_{n \geq 1}$}
We first state a standard   almost sure convergence result, established  with the help of the Robbins-Monro Theorem.
\begin{thm}[Robbins-Monro Theorem - Almost sure convergence of $(\tn)_{n \geq 1}$]\label{theo:robbinsmonro} Assume that  $a_n = a_1 n^{-a}$ with $a \in [1/2,1]$ and $a_1>0$, then the sequence $(\tn)_{n \geq 1}$ satisfies
$$
\tn \xrightarrow{ n \longrightarrow +\infty} \ta \qquad a.s.
$$
\end{thm}
\noindent
We refer to \cite{Duflo97} for a proof of this standard result.
The major drawback of Theorem \ref{theo:robbinsmonro}  is the lack of quantitative information about the rate of convergence. We are generally interested in a result of the form
$$ \E[(\tn-\ta)^{2q}] \leq r_{n,q},$$ where $r_{n,q}$ is referred to as the convergence rate of the algorithm. We emphasize that these bounds are more meaningful than a simpler almost sure convergence result because they make it possible to draw some quantitative comparisons among algorithms (in particular when looking at the rates of convergence and at the first order terms).
Deriving a such bound is at the cornerstone of our forthcoming study of $(\btn)_{n \geq 1}$.
\begin{thm}
\label{theo:momentp} Assume $\mathbf{H}_f$, if $a_n=a_1 n^{-a}$ with $a\in(0,1)$, then 
a large enough constant $C$ exists such that:
$$
\M_{n,4} := \E |\tn-\ta|^4 \leq C n^{-2 a}.
$$
More generally, for any integer $q \geq 1$, we have
$$
\exists K_q \geq 0 \quad \forall n \in \mathbb{N} \qquad \M_{n,2 q} =\E |\tn-\ta|^{2q}\leq K_q a_n^{q}.
$$
In particular, we can take:
$$
K_1 = \frac{16 c_1}{m} \qquad \text{and} \qquad K_2 = 8 \frac{c_2+\frac{8 c_2 c_1}{m} + e^{\frac{4 a_1^2 c_1}{2a-1}} a_1 }{m}
$$
where $c_1,c_2$ and $m$ are given in Lemma \ref{lem:prop_phi}.
\end{thm}
We emphasize that our result holds for $a  \in (0,1)$, instead of the classical assumption $a \in (1/2,1)$ in the Robbins-Monro Theorem (see for example the classical result on the \textit{dosage} in Theorem  1.4.26 and Proposition 1.4.27 of \cite{Duflo97}). 
We prove here that the common constraint $\sum_{n \ge 1} a_n^2 < +\infty$ is useless in the case of the recursive quantile algorithm to obtain a convergence result. This important and unusual result holds because of the boundedness of the noise from one iteration to another and is obtained with a subtle strong-convexification of the loss function with the help of the family of functions $V_q$. To the best of our knowledge, such improvement has only been observed  in the case of bounded noise in a general result of M\'etivier and Priouret (see \textit{e.g.} Theorem 1 of \cite{metivier1987theoremes}), 
 and was also stated in a simpler way in  Proposition 4.2 of \cite{benaim1999dynamics}. We emphasize that these works only state   a.s. asymptotic results.
Said differently, for the dosage issue and other companion problems (such as the mediane estimation for example), the assumption $a>1/2$ used in \cite{Duflo97,godichon2015} to assess convergence results on $(\tn)_{n \ge 1}$ is useless.

Second, we also observe that our upper bound is degraded according to $f(\ta) \longrightarrow 0$, which translates a flat area of the cumulative distribution function around $\ta$. Such a phenomenon is absolutely coherent with the CLT (not show here) known for the dosage problem.

\begin{proof}[Proof of Theorem \ref{theo:momentp}]
We start with the linearization of $(\tn)_{n\ge1}$. We define \begin{equation}\label{def:Phi}
\Phi(\theta) :=  \int_{\ta}^\theta \int_{\ta}^u f(s) \text{d} s \text{d}u.
\end{equation}
The function $\Phi$ is associated to the gradient descent involved by Equation \eqref{eq:algo}.
We point out that $\Phi \geq 0$ and satisfies:
\begin{equation}\label{eq:Phi-quadra}
\forall \theta \in \R \qquad 0 \leq \Phi(\theta) \leq \frac{\|f\|_{\infty}}{2} (\theta-\ta)^2.
\end{equation}
We shall remark that:
$$\begin{aligned}
\tnp &= \tn - a_{n} (F(\tn)-F(\ta)) + a_{n} \Delta M_{n+1}\\
& = \tn - a_{n} \Phi'(\tn) + a_{n} \Delta M_{n+1}, 
\end{aligned}$$
where the martingale increment is defined in by:
\begin{equation}
\label{eq:mart_inc}\Delta M_{n+1}:=F(\tn)-\mathds{1}_{X_{n+1}\le \tn}.
\end{equation}

We now follow the roadmap of \cite{Gadat-Panloup} and define the Lyapunov function:
\begin{equation}\label{def:Vq}
V_q(\theta) := \Phi(\theta)^q \exp(\Phi(\theta)), \quad \forall q\in\N.
\end{equation}
We first state some important properties of $V_q$, whose proofs are postponed to  Appendix \ref{app:tech_tn}.
\begin{lem}
\label{lem:prop_phi}
\begin{itemize}
\item[$i)$] A constant $m>0$ exists such that for any $\theta\in \R$ and $q\ge1:$
$$\Phi'(\theta)V_q'(\theta)\ge m V_q(\theta) \quad \text{with} \quad 
m = f(\ta) \left( \frac{f(\ta)}{4 \|f'\|_{\infty}^2} \wedge \frac{3 q}{8}\right)$$
\item [$ii)$] A constant $c_q>0$ that only depends on $q$ and $f$ exists such that:
\begin{equation*}
|V_q''(\theta)| \le  c_q(V_{q}(\theta)+V_{q-1}(\theta)).
\end{equation*}
where  
$C_{\alpha}= 9 f(\ta) \vee 4 \frac{\|f'\|_{\infty}^2}{f^3(\ta)}$ and  $c_q=1+q \|f\|_{\infty} + q[(q-1) \vee 2] C_{\alpha}$.

\item[$iii)$] For any $q\ge1$, a positive constant $C_q$ exists such that for any $\theta\in\R$
$$(\theta-\ta)^{2q}\le C_q \left(V_q(\theta)+V_{2q}(\theta)\right) \quad \text{where} \quad  C_q=\frac{4^q}{f(\ta)^q}  \wedge \frac{4^{2q} \|f'\|_{\infty}^{2q}}{f(\ta)^{4 q}}.$$
\end{itemize}
\end{lem}

Our strategy of proof relies on a recursive contraction from $V_q(\tn)$ to $ V_q(\tnp)$ using a Taylor expansion of $V_q(\tnp)$:
\begin{align*}
V_q(\tnp) &= V_q\left(\tn-a_{n} \Phi'(\tn)+a_{n} \Delta M_{n+1} \right) \\
& = V_q(\tn) - a_n\left( \Phi'(\tn) -  \Delta M_{n+1}\right) V_q'(\tn) + a_n^2\frac{(\Phi'(\tn)- \Delta M_{n+1})^2}{2} V_q''(\xi_{n+1}),
\end{align*}
where 
\begin{equation}\label{eq:def_xi}
\xi_{n+1} = \tn + \ell_{n+1} a_{n} \left(- \Phi'(\tn)+\Delta M_{n+1}\right),
\end{equation}
with $\ell_{n+1}$ is a random variable in $[0,1]$. \\
We first consider the drift term. Since $\Delta M_{n+1}$ is a martingale increment:
$$
\E \left[a_n(\Phi'(\tn)- \Delta M_{n+1})V'_q(\tn) \, \vert \mathcal{F}_n\right]  =  a_{n} V'_q(\tn)\Phi'(\tn),
$$
and Lemma \ref{lem:prop_phi} $i)$ yields
\begin{align}
\label{eq:drift_minoration}
\E \left[a_n(\Phi'(\tn)- \Delta M_{n+1})V'_q(\tn) \, \vert \mathcal{F}_n\right]>ma_nV_q(\tn).
\end{align}
We now pay a specific attention to the second order term. Lemma \ref{lem:prop_phi} $ii)$ yields:
\begin{align}
|V_q''(\xi_{n+1})| \le  c_q(V_{q}(\xi_{n+1})+V_{q-1}(\xi_{n+1})).
\label{eq:V''}
\end{align}
We are thus led to upper bound $\Phi(\xi_{n+1})$. Using a Taylor expansion near $\tn$, we obtain:
\begin{align*}
\lefteqn{\Phi(\xi_{n+1}) = \Phi\left(\tn+\ell_{n+1} a_{n} \left(\Delta M_{n+1}- \Phi'(\tn)\right)\right)}\nonumber\\
& \leq \Phi(\tn) + \Phi'(\tn) \ell_{n+1} a_{n} \left(\Delta M_{n+1}- \Phi'(\tn)\right) + \frac{\ell_{n+1}^2 a_{n}^2 \left(\Delta M_{n+1}- \Phi'(\tn)\right)^2}{2} \|\Phi''\|_{\infty}.
\end{align*}
We recall that $\ell_{n+1}$ is a random variable in $[0,1]$ and that  $\Phi'(\tn)$ and $ \Delta M_{n+1}$ are uniformly bounded by $1$. We thus deduce that:
\begin{equation*}
\Phi(\xi_{n+1})
\leq \Phi(\tn) +  2a_{n} +  2a_{n} ^2 \|f\|_{\infty} ,
\end{equation*}
which leads to, for a $n$ large enough:
\begin{equation*}
\Phi(\xi_{n+1})
\leq \Phi(\tn) +  3a_{n}.
\end{equation*}
Changing the constant $c_q$ from line to line, we are led to:
\begin{align}
V_q(\xi_{n+1})&\le \exp(\Phi(\tn))e^{3a_n}\left( \Phi(\tn) + 3a_{n} \right)^q\nonumber\\
&\le c_q\left(V_q(\tn)+a_{n}^q V_0(\tn))\right)\nonumber\\
&\le  c_q \left(V_q(\tn)+a_{n}^q (1+V_q(\tn))\right)\nonumber\\
&\le c_q\left(V_q(\tn)+a_{n}^q\right),
\label{eq:technique}
\end{align}
where we used that $\exp(\Phi(\tn))=V_0(\tn)\le C(1+V_q(\tn))$.\\
Combining Equations \eqref{eq:V''} and \eqref{eq:technique} we deduce that:
$$V_q''(\xi_{n+1})\le c_q\left( V_q(\tn)+V_{q-1}(\tn)+a_{n}^{q-1}\right).$$
Finally using again the fact that $\Phi'(\tn)$ and $ \Delta M_{n+1}$ are uniformly bounded, we obtain:
$$
\E \left(\frac{a_n^2( \Phi'(\tn)- \Delta M_{n+1})^2}{2} V_q''(\xi_{n+1}) \, \vert \F_n\right) \leq c_q a_{n}^2 \left( V_q(\tn)+V_{q-1}(\tn)+a_{n}^{q-1}\right).
$$
We conclude using \eqref{eq:drift_minoration} that:
\begin{align}
\E [ V_q(\tnp)] & = \E[\E[ V_q(\tnp) \, \vert \F_n] ] \nonumber\\
& \leq \E \left( V_q(\tn)\right) \left(1-a_{n} m  + c_q a_{n}^2 \right) + c_q a_{n}^2 \E \left( V_{q-1}(\tn)\right)  + c_q a_{n}^{q+1}. \label{eq:recurrence}
\end{align}
We deduce from the study of this recursive inequality given in Lemma \ref{lem:recursion_Vq} that for $n_0$ large enough
\begin{equation} \label{eq:Vqtn}
\forall q \geq 1, \quad \exists \, K_q \geq 0, \quad \forall n \ge n_0 \qquad \E [V_q(\tn)] \leq K_q \{a_{n}\}^q.
\end{equation}
In particular, we can choose
$$
K_1 = \frac{16 c_1}{m} \qquad \text{and} \qquad K_2 = 8 \frac{c_2+\frac{8 c_2 c_1}{m} + e^{A_1}{a_1}}{m}
$$
From Lemma \ref{lem:prop_phi}$iii)$ we have:
$$(\theta-\ta)^{2q}\le C_q \left(V_q(\theta)+V_{2q}(\theta)\right).$$
Taking the expectation at time $n$ and using Equation \eqref{eq:Vqtn}, we conclude that:
$$
\exists A_q \geq 0 \quad \forall n \in \mathbb{N} \qquad \mathcal{M}_{n,2q} \leq A_q a_n^q.
$$
\end{proof}
 
\subsection{Cesaro averaging}\label{sec:averaging_quantile}
 
This section is devoted to the proof of Theorem \ref{theo:rate_averaging} . The proof relies on a spectral analysis of the algorithm, close to the strategy developed in \cite{Gadat-Panloup}.

We introduce 
$Z_n = (\tn-\ta,\btn-\ta)$, which gives the joint evolution of the Robbins Monro algorithm and its averaged version.
\begin{prop}
\label{prop:lin_btn}
The sequence $(Z_n= (\tn-\ta,\btn-\ta)^T)_{n\ge0}$ satisfies:
\begin{equation}\label{eq:znpzn}
\Znp = A_n \Zn + a_{n} \left( \begin{matrix}
\Delta M_{n+1} \\ \frac{1}{n+1} \Delta M_{n+1}
\end{matrix} \right) + a_{n} \left( \begin{matrix}
 R_{n} \\ \frac{1}{n+1} R_{n}
\end{matrix} \right),
\end{equation}
where:
\begin{itemize}
\item[i)] $A_n$ translates the linearization of the algorithm around $(\ta,\ta)$ at step $n$:
$$
A_n = \left( \begin{matrix}
1 - a_{n} f(\ta) & 0 \\
\frac{1-a_n f(\ta)}{n+1} & 1-\frac{1}{n+1}
\end{matrix}
\right).
$$
\item[ii)] The martingale increment $\Delta M_{n+1}$ defined in \eqref{eq:mart_inc} satisfies:
$$
\left|\E \left[\{\Delta M_{n+1}\}^2\vert \Fn\right]-\alpha(1-\alpha)\right| \leq \|f\|_{\infty} |\ta-\tn|.
$$
\item[iii)] The rest term $R_n$ is an $\mathcal{F}_n$ measurable random variable such that:
\begin{equation}\label{eq:reste_1}
|R_{n} | \leq \frac{L}{2} |\tn-\ta|^{2}.
\end{equation}
\end{itemize}
\end{prop}

\begin{proof}
\textbf{Proof of $i)$}: 
We write the two sequences $(\tn)_{n \geq 1}$ and $(\btn)_{n \geq 1}$ as
$$
(\tnp-\ta) = (\tn-\ta) - a_{n} [F(\tn)-F(\ta)] + a_{n} \Delta M_{n+1},
$$
and 
$$\begin{aligned}
\btnp &=\frac{n}{n+1}\btn+\frac{1}{n+1}\tnp\\
&=\left(1-\frac{1}{n+1}\right) \btn + \frac{1}{n+1}\Bigl(\tn - a_{n} (F(\tn)-F(\ta)) + a_{n}\Delta M_{n+1}\Bigr).
\end{aligned}$$
We then use a linear approximation of $F$ around $\ta$ and write that:
$$
F(\tn)-F(\ta) =f(\ta) (\tn-\ta) + \int_{\ta}^{\tn} [f(u) - f(\ta)] \text{d}u.
$$
%\seb{ici}
Finally, we can write the joint evolution of $(\tn,\btn)_{n \geq 1}$ as a perturbed linear recursive sequence given by
\eqref{eq:znpzn}.
\\
\textbf{Proof of $ii)$}:
We study the variance of the martingale increment defined in \eqref{eq:mart_inc}:
\begin{align*}
\E \left[\{\Delta M_{n+1}\}^2\vert \Fn\right]  &= F(\tn)(1-F(\tn)) \\
& = \alpha (1-\alpha ) +    \int_{\ta}^{\tn} f(u) [1-2 F(u)] \text{d}u.
\end{align*}
We then deduce that:
$$
\left|\E \left[\{\Delta M_{n+1}\}^2\vert \Fn\right]-\alpha(1-\alpha)\right| \leq \|f\|_{\infty} |\ta-\tn|.
$$
\textbf{Proof of $iii)$}: The rest term is an $\mathcal{F}_n-$measurable random variable given by
$$R_{n}=\int_{\ta}^{\tn} [f(u) - f(\ta)] \text{d}u.$$
Using that $f$ is $L$-Lipschitz, the rest term $R_{n}$  satisfies:
$$
|R_{n} | \leq \frac{L}{2} |\tn-\ta|^{2}.
$$
\end{proof}

\begin{proof}[Proof of Theorem \ref{theo:rate_averaging} $i)$]
The scheme of the proof is to use the linearization to identify the negligible terms in the recursion. To do so, we diagonalize the matrix $A_n$ and perform a change of base.
The eigenvalues of $A_n$ are $ 1-a_{n} f(\ta)$ and $1-\frac{1}{n+1}$ and we verify that:
$$
A_n = \left( \begin{matrix}
1 & 0 \\
\epsilon_{n} & 1
\end{matrix} \right)
\left( \begin{matrix}
1 - a_{n} f(\ta) & 0 \\
0 & 1-\frac{1}{n+1}
\end{matrix}
\right)
\left( \begin{matrix}
1 & 0 \\
-\epsilon_{n} & 1
\end{matrix} \right) \quad \text{with} \quad \epsilon_{n}:=\frac{1-a_{n} f(\ta)}{1-a_{n} (n+1) f(\ta) }.
$$
We introduce
\begin{equation}\label{def:tZn}
\tZn = \left( \begin{matrix}
1 & 0 \\
-\epsilon_{n} & 1
\end{matrix} \right) \Zn = \left(
\begin{matrix}
\tn-\ta \\ - \epsilon_{n}(\tn-\ta)+(\btn-\ta),
\end{matrix}
\right),
\end{equation}
and we denote by $\tZnd$ the second component of $\tZn$. Equation \eqref{eq:znpzn} yields:
\begin{equation}
\label{eq:rec-tZn}
\tZnpd=\left(1-\frac{1}{n+1}\right)\tZnd+a_n\Delta M_{n+1}\left(\frac{1}{n+1}-\epsilon_{n+1}\right) +\mathcal{R}_n,\end{equation}
where:
\begin{equation}
\label{eq:restes}
\mathcal{R}_n=\omega_n(\tn-\ta)+\left(\frac{1}{n+1}-\epsilon_{n+1}\right)a_nR_n \quad \text{with} \quad \omega_n := (\epsilon_{n+1}-\epsilon_{n})(1-a_{n} f(\ta)).
\end{equation}
We now compute the quadratic expansion using the fact that $(\Delta M_{n+1})_{n \ge 1}$ is a sequence of martingale increments:
\begin{align}
\E (\{\tZnpd\}^2) =&\left(1-\frac{1}{n+1}\right)^2  \E (\{\tZnd\}^2)+a_{n}^2 \left( \epsilon_{n+1}-\frac{1}{n+1}\right)^2 \E( \{\Delta M_{n+1}\}^2) \nonumber\\
&+\E(\mathcal{R}_n^2)+2\left(1-\frac{1}{n+1}\right)\E(\mathcal{R}_n\tZnd).\label{eq:rec_tZnd}
\end{align}

\underline{Step 1 : First approximation of $\E (\{\tZn^{(2)}\}^2)$.}\\
We denote $u_n = \E (\{\tZn^{(2)}\}^2)$ and use Cauchy-Schwarz inequality on the last term to deduce that:
\begin{align*}
u_{n+1} \le &\left(1-\frac{1}{n+1}\right)^2  u_n+a_{n}^2 \left( \epsilon_{n+1}-\frac{1}{n+1}\right)^2 \E (\{\Delta M_{n+1}\}^2) \\
&+\E(\mathcal{R}_n^2)+2\left(1-\frac{1}{n+1}\right)\sqrt{u_n\E(\mathcal{R}_n^2)}.
\end{align*}

We study in Lemma \ref{lem:t1} and \ref{lem:t2} the behaviour of all the terms involved in the previous decomposition.  Lemma \ref{lem:t1} concerns the variance of the martingale increments from which a term of order $n^{-2}$ arises, while Lemma \ref{lem:t2} proves that the other terms are negligible. 

\begin{lem}\label{lem:t1} 
For every $\delta>0$, a constant $C_\delta$ exists such that for all $n$:
$$a_{n}^2 \left( \epsilon_{n+1}-\frac{1}{n+1}\right)^2 \E (\{\Delta M_{n+1}\}^2) \le
\frac{\alpha(1-\alpha)}{(n+1)^2 f(\ta)^2} + \frac{2\alpha (1-\alpha )a_1}{f(\ta)^3 n^{(3-a)}} +C (n^{-2-a/2}\vee n^{-(3-a)-\delta}) .$$
\end{lem}
\begin{lem}
\label{lem:t2}There exists a constant $C$ (independent from $n$) such that for all $n$:
$$\E(\mathcal{R}_n^2)\le C \left(n^{-(4-a)}\vee n^{-(2+2a)}\right).$$
\end{lem}
The proof of these results is postponed in Appendix \ref{app:lem_quantile}. We deduce from these lemmas that some constant $\Gamma$ exists such that:
\begin{align*}
u_{n+1} & \leq \left(1-\frac{1}{n+1}\right)^2  u_n 
+ \frac{\alpha(1-\alpha)}{(n+1)^2 f(\ta)^2} +\Gamma \left[n^{-(2+a/2)} \vee n^{-(3-a)} \right]\nonumber \\
&\qquad + \Gamma \sqrt{u_n} \left[n^{-(2-a/2)}\vee n^{-(1+a)}\right]. \label{eq:recursion1}
\end{align*}
\noindent
We now apply Lemma \ref{lem:recurrence_finale_3} %to Equation \eqref{eq:recursion1} 
in order to derive the convergence rate of the algorithm. With the notations introduced in Lemma \ref{lem:recurrence_finale_3} applied with $\gamma=1$, $A_1=0$ and:
$$
C=\frac{\alpha(1-\alpha)}{f(\ta)^2} \quad \text{and} \quad r_2= (2+\frac{a}{2}) \wedge (3-a) \quad \text{and}\quad r_1= (2-\frac{a}{2})\wedge(1+a).
$$
We then obtain that a large enough $\Gamma_a$ exists such that:
\begin{equation}
\label{eq:maj_un1}
u_n  = \E (\{\tZn^{(2)}\}^2) \leq \frac{\alpha(1-\alpha)}{f(\ta)^2 n} + \Gamma_a n^{-r},
\end{equation}
where $r= \left( \frac{3}{2}-\frac{a}{2}\right)\wedge \left(\frac{1}{2}+a\right),$ as $a\in[1/2,1)$.
\noindent\underline{Step 2 : Refinement of the previous bound.}\\
Here however the second order term is not optimal and to improve our bound, we  refine our analysis with a precise study of the covariance term:
$$\E(\mathcal{R}_n\tZnd)=\omega_n\E((\tn-\ta)\tZnd) + a_n\left(\frac{1}{n+1}-\epsilon_{n+1}\right) \E(R_n \tZnd).$$
\begin{lem}
\label{lem:cov}
For all $n\ge1$, we have:
$$\E(\mathcal{R}_n\tZnd) \le   \frac{(a+1)\alpha(1-\alpha)}{f(\ta)^3a_1}n^{-3+a} + \frac{L\alpha(1-\alpha)a_1}{2f(\ta)^3} n^{-3/2-a}+ o(n^{-3+a}\vee n^{n^{-3/2-a}}).$$
\end{lem}

Gathering these bounds in \eqref{eq:rec_tZnd} we obtain that:
\begin{align*}
\E (\{\tZnpd\}^2) \le &\left(1-\frac{1}{n+1}\right)^2  \E (\{\tZnd\}^2)+\frac{\alpha(1-\alpha)}{(n+1)^2 f(\ta)^2}\\& +\kappa \frac{\alpha(1-\alpha)}{f(\ta)^3}\left( n^{-(3/2+a)} \vee n^{-(3-a)}\right)+ \varepsilon_n,
\end{align*}
with $\kappa =\max(L K_2 a_1, 2(a+1)/a_1, 2a_1)$ and $\varepsilon_n$ a negligible term with respect to $\left( n^{-(3/2+a)\wedge(3-a)}\right)$.\\
In particular, there exists $n_0$ large enough such that for all $n\ge n_0$, 
\begin{align*}
\E (\{\tZnpd\}^2) \le &\left(1-\frac{1}{n+1}\right)^2  \E (\{\tZnd\}^2)+\frac{\alpha(1-\alpha)}{(n+1)^2 f(\ta)^2}\\& +2\kappa \frac{\alpha(1-\alpha)}{f(\ta)^3}\left( n^{-(3/2+a)} \vee n^{-(3-a)}\right).
\end{align*}
We finally apply Lemma \ref{lem:recurrence_finale_precis} and deduce that a large enough $\Gamma_a$ exists such that:
\begin{equation*}
u_n \leq \frac{\alpha(1-\alpha)}{f(\ta)^2 n} + \kappa' \frac{\alpha(1-\alpha)}{f(\ta)^3} n^{-\left( 2-a\right)\wedge \left(\frac{1}{2}+a\right)} + \Gamma_a n^{-\left( 3-a\right)\wedge \left(\frac{3}{2}+a\right)}
\end{equation*}
where
\begin{equation}
\label{def:kappa'}
\kappa'=\frac{2\kappa}{\left( 3-a\right)\wedge \left(\frac{3}{2}+a\right)}.
\end{equation}
The optimal choice of $a$ is to $a=3/4$ and in this case 
\begin{equation}
\label{eq:maj_unopt}
u_n \leq \frac{\alpha(1-\alpha)}{f(\ta)^2 n} +   \frac{8\kappa \alpha(1-\alpha)}{7f(\ta)^3} n^{-5/4} + \Gamma_a n^{-9/4}.
\end{equation}% is given by: 
\underline{Step 3: Return to $\E(\{\btn-\ta\}^2)$.}\\
%To conclude the proof, we come back to $\E(\{\btn-\ta\}^2)$.
From the definition of $\tZn^{(2)}$ we obtain that:
$$\E(\{\btn-\ta\}^2)=\E \{\tZn^{(2)}\}^2+2\epsilon_n\E(\tZn^{(2)}(\tn-\ta))+\epsilon_n^2\E((\tn-\ta)^2).
$$
Therefore, it remains to prove that the two last terms are negligible with respect to $n^{-r}$.
From Theorem \ref{theo:momentp}, we know that a large enough $\Gamma$ exists such that
$$\epsilon_n^2\E((\tn-\ta)^2)\leq  \frac{\Gamma}{n^2a_n},$$
and the Cauchy-Schwarz inequality yields (up to a modification of $\Gamma$):
$$\epsilon_n\E(\tZn^{(2)}(\tn-\ta))\leq \Gamma n^{-3/2-a/2}.$$
\end{proof}

\begin{proof}[Proof of Theorem \ref{theo:rate_averaging}-$ii)$] The proof is deferred to Appendix \ref{app:lem_quantile}.
\end{proof}

\section{Embedded averaging super-quantile algorithm}\label{sec:superquantile}
The purpose of this paragraph is to prove Theorem \ref{theo:rate_vtn},  \textit{i.e.}, we aim to derive a non-asymptotic result on the sequence $(\bvtn)$ (see the definition in Equation \eqref{eq:algo}).

\subsection{Linearization of the embedded algorithm}
To study the behavior of Algorithm \eqref{eq:algo}, we follow the roadmap of Section \ref{sec:quantile} and write a linear approximation of $(\bvtn)_{n \geq 1}$.
To this end, we define:
\begin{equation}\label{def:Q}
Q: \theta \longmapsto \E \left[X \mathbf{1}_{X \geq \theta}\right],
\end{equation}
and
\begin{equation}
\label{eq:mart_super}\Delta N_{n+1}:=X_{n+1}\mathds{1}_{X_{n+1}\geq \btn}-Q(\btn).
\end{equation}
We shall prove the following result.

\begin{prop}\label{prop:linearization} Assume $\mathbf{H}_f$, for any integer $n$, one has
\begin{align}
\bvtnp-\vta  & = (\bvtn-\vta)\left(1-b_n\right) - \frac{b_{n}}{1-\alpha}(\btn-\ta)  \ta f(\ta) \nonumber\\
&\quad\quad + \frac{b_{n}}{1-\alpha} Q''(\xi_{n}) \frac{(\btn-\ta)^2}{2}  + \frac{b_{n}}{1-\alpha} \Delta N_{n+1},%\label{eq:recvtn}
\end{align}
where $\xi_{n} \in (\btn,\ta)$,$V_{\alpha}$ is defined in \eqref{def:Valpha} and $ \Delta N_{n+1}$ verifies:
\begin{equation}
\label{eq:martingale_variance}
\left|\E [\{ \Delta N_{n+1}\}^2 \vert \Fn] - V_{\alpha}\right| \leq C |\btn-\ta|
\quad \text{and} \quad 
  | \E 
 \{ \Delta N_{n+1}\}^2 - V_{\alpha}| \leq C \sqrt{\MM_{n,2}}.
\end{equation}
\end{prop}

\begin{proof}
We observe that:
\begin{align*}
\bvtnp-\vta&=\bvtn-\vta+b_{n}\left[\un_{X_{n+1} \geq \btn}\frac{X_{n+1}}{1-\alpha}-\bvtn\right] \\
&=
\bvtn-\vta+ \frac{b_{n}}{1-\alpha}\left[X_{n+1}(\un_{X_{n+1} \geq \btn}-\un_{X_{n+1} \geq \ta})\right]+b_{n}\left[\un_{X_{n+1} \geq \ta}\frac{X_{n+1}}{1-\alpha}-\bvtn\right]. 
\end{align*}
According to Equations \eqref{def:Q} and \eqref{eq:mart_super}, we obtain that:
$$
\bvtnp-\vta=\left( \bvtn-\vta\right) \left(1-b_{n}\right) + \frac{b_{n}}{1-\alpha} [Q(\btn)-Q(\ta)] + \frac{b_{n}}{1-\alpha} \Delta N_{n+1}.
$$
We use a second order Taylor expansion on $Q$ near $\ta$
using the relationships:
$$
Q(\theta) = \int_{\theta}^{+\infty} x f(x) \text{d}x, \qquad Q'(\theta) = -\theta f(\theta) \quad \text{and} \quad Q''(\theta) = -f(\theta) - \theta f'(\theta).
$$
In particular,  $\mathbf{H}_f$ implies that   $\|Q''\|_{\infty} < + \infty$. 
We obtain that:
\begin{align*}
\bvtnp-\vta  & = (\bvtn-\vta)\left(1-b_n\right) - \frac{b_{n}}{1-\alpha}(\btn-\ta)  \ta f(\ta) \\
&\quad\quad + \frac{b_{n}}{1-\alpha} Q''(\xi_{n}) \frac{(\btn-\ta)^2}{2}  + \frac{b_{n}}{1-\alpha} \Delta N_{n+1},
\end{align*}
where $\xi_{n}\in[\ta,\btn]$. 
\medskip

Now, we compare the variance of the martingale increment and $V_{\alpha}$:
\begin{align*}
\E[ \{\Delta N_{n+1}\}^2\vert \Fn] %& = \E \left[ X_{n+1} \un_{X_{n+1} \geq \btn} \vert \Fn\right]^2 - \left( \E \left[ X_{n+1} \un_{X_{n+1} \geq \btn} \vert \Fn\right] \right)^2 \\
& = \E  \left[ X_{n+1}^2 \un_{X_{n+1} \geq \btn} \vert \Fn\right]  - \left( \E \left[ X_{n+1} \un_{X_{n+1} \geq \btn} \vert \Fn\right] \right)^2 \\
& =  \int_{\btn}^{\infty} x^2 f(x)  dx  - \left[ \int_{\btn}^{+\infty} x f(x) dx\right]^2 \\
& = \int_{\ta}^{+\infty} x^2 f(x) dx -  \left( \int_{\ta}^{+ \infty} x f(x) dx\right)^2  + \delta_{n+1},
\end{align*}
where:
\begin{align*}
\delta_{n+1}& =  \int_{\btn}^{\ta} x^2 f(x) dx -\left[ \int_{\btn}^{+\infty} x f(x) dx\right]^2  + \left( \int_{\ta}^{+ \infty} x f(x) dx\right)^2  \\
& =  \int_{\btn}^{\ta} x^2 f(x) dx  + \left( \int_{\ta}^{\btn} x f(x) dx \right) \left( 2 \int_{\ta}^{+ \infty} x f(x) dx  + \int_{\btn}^{\ta} x f(x) dx \right).
\end{align*}

Using $\mathbf{H}_f$, $x\mapsto x^2f(x)$ is bounded so that a large enough constant $C$ exists such that:
$$
|\delta_{n+1}| \leq C |\btn-\ta|.
$$
Computing the whole expectation and using the Cauchy-Schwarz inequality, we have that:
$$
\left| \E \{\Delta N_{n+1}\}^2 - V_{\alpha} \right| \leq C \sqrt{\MM_{n,2}}.
$$
\end{proof}

\subsection{Spectral analysis and recursion}

We now study the sequence $(\bvtn)$ with the help of the previous linearization combining Propositions \ref{prop:linearization} and \ref{prop:lin_btn}.
We introduce the matrix
\begin{equation}\label{eq:matrix_mn}
B_n :=
\left( 
\begin{matrix}
1-a_n f(\ta) & 0 & 0 \\
\frac{1-a_n f(\ta)}{n+1} & 1-\frac{1}{n+1} & 0 \\
0 & - \frac{\ta f(\ta)}{1-\alpha} b_n & 1-b_n
\end{matrix}
\right).
\end{equation}
We emphasize that the eigenvalues of $B_n$ are 
$$
Sp(B_n) = \left\{1-a_n f(\ta);1-\frac{1}{n+1};1-b_n \right\}.
$$
Obviously, $B_n$ may be reduced to a diagonal form as soon as $b_n \notin\{a_nf(\ta), \frac{1}{n+1}\}$. In particular, we obtain the following decomposition (stated without any proof).
\begin{prop}\label{eq:decomposition_Bn}
For any integer $n$, and if $b_n \notin\{a_nf(\ta), \frac{1}{n+1}\}$, we have
$$
B_n =P_n
D_n 
P_n^{-1}\quad\text{with}\quad D_n=\left( \begin{matrix}
1-a_n f(\ta) & 0 & 0 \\
0 & 1-\frac{1}{n+1} & 0 \\
0 & 0 & 1-b_n
\end{matrix}
\right),
$$
where $P_n$ is the matrix of change of basis:
$$P_n=\left( \begin{matrix}
1 & 0 & 0\\
\epsilon_n & 1 & 0 \\
\epsilon_n\delta_n & \kan & 1
\end{matrix} \right), \qquad P_n^{-1}=\left(\begin{matrix}
1 & 0 & 0\\
-\epsilon_n & 1 & 0 \\
-\epsilon_n \delta_n+\epsilon_n \kan& -\kan  & 1
\end{matrix} \right),$$
with:
$$
\epsilon_n = \frac{1-a_n f(\ta)}{1-a_{n}(n+1) f(\ta)}, \quad \delta_n = \frac{\ta f(\ta)}{(1-\alpha)}\left(\frac{a_n}{b_n}f(\ta)-1)\right)^{-1} , \quad
 \kan= \frac{\ta f(\ta) b_n (n+1)}{(1-\alpha) (1-b_n (n+1))}.
$$
\end{prop}
As pointed out above, $B_n$ may not be reduced to a diagonal form for \textit{any} sequence $b_n \longrightarrow 0$. Therefore, instead of using an exact spectral decomposition as it is the case in \cite{Gadat-Panloup}, the novelty of our work here is to handle a decomposition of  $B_n$ without an exact diagonal form. For this purpose, we introduce the invertible matrix $\widetilde{P}_n$ defined by:
$$
\widetilde{P}_n  = 
\left( \begin{matrix}
1 & 0 & 0\\
\epsilon_n & 1 & 0 \\
\epsilon_n\delta_n &0 & 1
\end{matrix} \right),
$$
and we verify that:
\begin{equation}\label{eq:presque_diagonal}
\widetilde{P}_n^{-1} B_n \widetilde{P}_n = D_n - b_n \frac{\ta f(\ta)}{1-\alpha}
\underbrace{\left( \begin{matrix}
0 & 0 & 0\\
0 & 0 & 0 \\
0 & 1 & 0
\end{matrix} \right)}_{:=E_{3,2}} = D_n - b_n \frac{\ta f(\ta)}{1-\alpha} E_{3,2}.
\end{equation}
\begin{rmq}
The matrix $\widetilde{P}_n$ corresponds to setting $\kan=0$ which removes the singularity issue since $\kan\to\infty$ for $b_n$ close to $1/(n+1)$.
\end{rmq}
Equation \eqref{eq:presque_diagonal} may be used as follows: we introduce the vector defined by:
$$
\con:= \left( \begin{matrix}
\tn-\ta \\ \btn-\ta \\ \bvtn-\vta
\end{matrix}\right),
$$
and we remark that Proposition \ref{prop:linearization} yields:
$$
\conp = B_n \con+\left( \begin{matrix}a_n \Delta M_{n+1}
\\\frac{a_n}{n+1} \Delta M_{n+1} \\ \frac{b_n}{(1-\alpha)} \Delta N_{n+1}
\end{matrix}\right) + \left( \begin{matrix}
a_n R_{n} \\ \frac{a_n}{n+1} R_{n} \\ b_n \widetilde{R}_{n}
\end{matrix}\right) \quad \text{where} \quad \widetilde{R}_{n} :=\frac{Q''(\xi_{n})}{(1-\alpha)} \frac{(\btn-\ta)^2}{2}.
$$
Using  $\widetilde{P}_n$, we translate the spectral information on $B_n$ to the new vector:
$
\tcon := \widetilde{P}_n^{-1} \con.
$
Equation \eqref{eq:presque_diagonal} yields:
\begin{align*}
\tconp & = \widetilde{P}_{n+1}^{-1} \conp\\
&=
\left(D_n - b_n \frac{\ta f(\ta)}{1-\alpha} E_{3,2}\right) \tcon 
+(\widetilde{P}_{n+1}^{-1}\widetilde{P}_n -I_3)\left(D_n - b_n \frac{\ta f(\ta)}{1-\alpha} E_{3,2} \right) \tcon \\
& + \widetilde{P}_{n+1}^{-1}\left( \begin{matrix}
a_n R_{n} \\ \frac{a_n}{n+1} R_{n} \\ b_n \tilde{R}_{n}
\end{matrix}\right) +\widetilde{P}_{n+1}^{-1}\left( \begin{matrix}a_n \Delta M_{n+1}
\\\frac{a_n}{n+1} \Delta M_{n+1} \\ \frac{b_n}{1-\alpha} \Delta N_{n+1}
\end{matrix}\right).
\end{align*}
We define:
\begin{equation}\label{def:eta}
\eta_n := \epsilon_n \delta_n - \epsilon_{n+1} \delta_{n+1}, 
\end{equation}
and verify that:
\begin{align*}
(\widetilde{P}_{n+1}^{-1}\widetilde{P}_n -I_3)
\left(D_n - b_n \frac{\ta f(\ta)}{1-\alpha} E_{3,2} \right) 
%&= \begin{pmatrix}
%0&0&0\\
%\epsilon_n-\epsilon_{n+1}&0&0\\
%\eta_n  &0&0
%\end{pmatrix} \left(D_n_n - b_n \frac{\ta f(\ta)}{1-\alpha} E_{3,2} \right) \\
%&  =  \begin{pmatrix}
%0&0&0\\
%\epsilon_n-\epsilon_{n+1}&0&0\\
%\eta_n  &0&0
%\end{pmatrix} D_n \\
& =  \begin{pmatrix}
0&0&0\\
(1-a_n f(\ta))(\epsilon_n-\epsilon_{n+1})&0&0\\
(1-a_n f(\ta)) \eta_n  &0&0
\end{pmatrix}. 
\end{align*}
Considering the third coordinate of the sequence $(\tcon)_{n \ge 1}$, we obtain the recursion:
\begin{align*}
\tconp^{(3)} & = \left(1-b_n\right) \tconn{3} -  b_n \frac{\ta f(\ta)}{1-\alpha}\tconn{2} 
+   (1-a_n f(\ta)) \eta_n   \tconn{1}  \\
& - \epsilon_{n+1} \delta_{n+1} a_n  R_{n} + b_n\widetilde{R}_{n} - \epsilon_{n+1} \delta_{n+1} a_n  \Delta M_{n+1}+ \frac{b_n}{(1-\alpha)} \Delta N_{n+1}.%\nonumber
\end{align*}
All negligible $\mathcal{F}_n$-measurable terms are gathered in $U_n$ defined by:
\begin{equation}
\label{def:U_n}
U_n :=  -  b_n \frac{\ta f(\ta)}{1-\alpha}\tconn{2} 
+   (1-a_n f(\ta)) \eta_n   \tconn{1}  - \epsilon_{n+1} \delta_{n+1} a_n  R_{n} + b_n\widetilde{R}_{n},
\end{equation}
and we regroup the martingale increments into $\Delta\mathcal{V}_{n+1}$ defined by:
$$
\Delta\mathcal{V}_{n+1}=\epsilon_{n+1} \delta_{n+1} a_n  \Delta M_{n+1}+ \frac{b_n}{(1-\alpha)} \Delta N_{n+1}.$$
Then the recursion reads:
\begin{equation}\label{eq:recursion_on}\tconp^{(3)}  = \left(1-b_n\right) \tconn{3}+ U_n +\Delta\mathcal{V}_n.
\end{equation}
The next result gives a recursive inequality on 
$W_n=\E\left(\left\{ \tconn{3}\right\}^2\right).$
\begin{prop}
\label{prop:W_n}
Assume that $\mathbf{H}_f$ and $\mathbf{H}_{(a_n,b_n)}$ hold, then $W_n=\mathbb{E} \{\tconn{3}\}^2$ satisfies:
\begin{equation}
\label{eq:recW}
\begin{aligned}
W_{n+1} &\leq (1-b_n)^2 W_n + \frac{b_n^2V_\alpha}{(1-\alpha)^2}+ C \sqrt{W_n} \left( n^{-(2-a/2)}\vee n^{-(b+1/2)}\right)\\
&\qquad+ C \left(n^{-(2b+a/2)}\vee n^{-(3-a)}\vee  n^{-(2b+1)}\right),
\end{aligned}
\end{equation}
where $V_\alpha$ is given in \eqref{def:Valpha} and $C>0$.
\end{prop}
\begin{proof}
A direct square expansion of \eqref{eq:recursion_on} %associated with the Young inequality $ x y \leq x^2/2+y^2/2$
yields:
\begin{align*}
W_{n+1}& \le\left(1-b_n\right)^2 W_n+ \E(U_n^2)+2\left(1-b_n\right)  \E[ \tconn{3} U_n] +\E(\Delta\mathcal{V}_{n+1}^2)
%\underbrace{\epsilon_{n+1}^2 \delta_{n+1}^2 a_n^2 \E[ \{\Delta M_{n+1}\}^2] + \frac{b_n^2}{(1-\alpha)^2} \E[\{\Delta N_{n+1}\}^2 ] - 2 \frac{\epsilon_{n+1} \delta_{n+1} a_n b_n }{1-\alpha} \E[\Delta M_{n+1} \Delta N_{n+1}]}_{:=\mathcal{V}_n}.
\end{align*}

The Cauchy-Schwarz inequality entails
$2\left(1-b_n\right)  \E[ \tconn{3} U_n] \le 2 | 1-b_n| \sqrt{\E [U_n^2]} \sqrt{W_n}
$. The next lemmas derive some bounds on $\E(U_n^2)$ and $\E(\Delta\mathcal{V}_{n+1}^2)$. The proofs are given in Appendix \ref{app:calcul_sq}.
\begin{lem}\label{lem:tec_termes_carres}
Assume that $\mathbf{H}_f$ and $\mathbf{H}_{(a_n,b_n)}$ hold, then:
$$
%\eta_n^2 \E(\{\tconn{1}\}^2)  + b_n^2  \E(\{\tconn{2}\}^2) +  \epsilon_{n+1}^2 \delta_{n+1}^2 a_n^2 \E[ R_{n+1}^2] +b_n^2\E(\tilde{R}_{n+1}^2) 
\E(U_n^2)\lesssim n^{-(2b+1)}.
$$
Moreover, in the special case where $b_n=b_1(n+1)^{-1}$, a positive $C$ exists such that:
$$
\E(U_n^2)\le 4\frac{\alpha(1-\alpha)}{f(\ta)^2 }b_1^2n^{-3}  + C n^{-(4-a)\wedge(\frac{5}{2}+a)}.
$$
\end{lem}
\begin{lem}
\label{lem:tec_termes_martingales}
Assume $\mathbf{H}_f$ and $\mathbf{H}_{(a_n,b_n)}$ hold, then:
$$
\E(\Delta\mathcal{V}_{n+1}^2)=b_n^2\frac{V_\alpha}{(1-\alpha)^2}  + O\left(b_n^2a_n^{1/2}\vee n^{-(3-a)}\right).$$
\end{lem}
\noindent
We apply Lemma \ref{lem:tec_termes_martingales} and  Lemma \ref{lem:tec_termes_carres} and obtain that a large enough $C>0$ exists such that:
\begin{align*}
W_{n+1} \leq (1-b_n)^2 W_n + \frac{b_n^2V_\alpha}{(1-\alpha)^2}+ C \sqrt{W_n}  n^{-(b+1/2)} + C n^{-(2b+a/2)\wedge (3-a)}.
\end{align*}
\end{proof}
We conclude this section with the proof of Theorem \ref{theo:rate_vtn}.

\begin{proof}[Proof of Theorem \ref{theo:rate_vtn}]
~\\\noindent\underline{Case $b_n=b_1(n+1)^{-b}$ with $b\in(1/2,1)$.}
We apply Lemma \ref{lem:recurrence_finale_2} to the recursion obtained in Proposition \ref{prop:W_n}, with 
$$C =\frac{V_\alpha}{(1-\alpha)^2}, \quad r_1= b+1/2  \quad \text{and} \quad r_2=(2b+a/2)\wedge (3-a).$$
Then, we easily verify that the largest possible value of $p$ in Lemma \ref{lem:recurrence_finale_2} is 
$$
p=r_1-\frac{b}{2} \wedge r_2-b\wedge 2b \wedge 1 = \frac{b+1}{2}.
$$
Since $b<p$, we shall apply Lemma \ref{lem:recurrence_finale_2} and we obtain that a large enough $\Gamma_a$ exists such that:
$$
W_n \leq \frac{V_\alpha}{2(1-\alpha)^2}b_n + \Gamma n^{-\frac{b+1}{2}}.
$$
\medskip
~\\\noindent\underline{Case $b_n=b_1(n+1)^{-1}$.}
In this case we use  Lemma \ref{lem:tec_termes_carres} to refine the recursion obtained in Proposition \ref{prop:W_n} into:
\begin{align*}
W_{n+1} &\leq (1-\frac{b_1}{n+1})^2 W_n + \frac{b_1^2V_\alpha}{(1-\alpha)^2(n+1)^2}+ C \left(n^{-(2+a/2)\wedge (3-a)}\right)\\
&\qquad +2 \sqrt{W_n}(1-b_n)\sqrt{\frac{4\alpha(1-\alpha)}{f(\ta)^2 }}b_1n^{-3/2} (1-b_n)\sqrt{1+Cn^{-(1-a)\wedge(a-1/2)}}\\
&\leq (1-\frac{b_1}{n+1})^2 W_n + \frac{b_1^2V_\alpha}{(1-\alpha)^2(n+1)^2}+ C \left(n^{-(2+a/2)\wedge (3-a)}\right)\\
&\qquad +\sqrt{W_n}\left[4b_1\frac{\sqrt{\alpha(1-\alpha)}}{f(\ta)}n^{-3/2}+Cn^{-(5/2-a)\wedge(a+1)}\right].\\
\end{align*}
We now use Lemma \ref{lem:recurrence_finale_3} with $\gamma=b_1$, $C=\frac{b_1^2V_\alpha}{(1-\alpha)^2}$, $A_1=4 b_1\sqrt{\frac{\alpha(1-\alpha)}{f(\ta)^2 }}$, $r_1=(5/2-a)\wedge(a+1)$ and $r_2=(2+a/2)\wedge(3-a)$. 
If we choose $\rho=(1+\frac{a}{2})\wedge (2-a)$
and
\begin{align*}
C_{\alpha,b_1}&=\left(\frac{A_1}{2(2\gamma-1)} + \sqrt{\frac{A_1^2}{4(2\gamma-1)^2}+\frac{C}{2\gamma-1}}\right)^2\\
&= \frac{4b_1^2\alpha(1-\alpha)}{(2b_1-1)^2f(\ta)^2}\left[1+\sqrt{1+\frac{V_\alpha f(\ta)^2(2b_1-1)}{4\alpha(1-\alpha)^3}} \right]^2,
\end{align*}
then we obtain that for any $b_1>\rho-1/2$:
$$W_n\le \frac{C_{\alpha, b_1}}{n}+\Gamma n^{-(1+a/2)\wedge (2-a)},
$$
and  this inequality is made optimal when $a=\frac{2}{3}$, which entails a  second order term equals to $n^{-4/3}$ when $b_1>5/6$.

\medskip
~\\\noindent\underline{Conclusion:}
To conclude the proof, we come back to $\E(\{\bvtn-\vta\}^2)$.
From the definition of $\tcon^{(3)}$ we obtain 
$$\E(\{\bvtn-\vta\}^2)=\E( \{\tcon^{(3)}\}^2)+2\epsilon_n\delta_n\E(\tcon^{(3)}(\tn-\ta))+\epsilon_n^2\delta_n^2\E((\tn-\ta)^2).
$$
Therefore, it remains to prove that the two last terms are negligible with respect to the first one.
From Theorem \ref{theo:momentp}, we know that a large enough $C$ exists such that
$$\epsilon_n^2\delta_n^2\E((\tn-\ta)^2)\leq  Cn^{-2-2b+3a},$$
and the Cauchy-Schwarz inequality yields (up to a modification of $C$):
$$\epsilon_n\delta_n\E(\tcon^{(3)}(\tn-\ta))\leq C n^{-(1+\frac{3}{2}b-\frac{3}{2}a)}.$$
When $b=1$, an easy comparison leads to:
$$\frac{n^{-4+3a}+ n^{-(5/2-3/2a)}}{n^{-(1+a/2)\wedge (2-a)}}\longrightarrow0.$$
When $b\in(1/2,1)$ and $a<b$, we can verify that both terms are negligible when compared to $n^{-(1+a)/2}$.
\end{proof}
\section{Central Limit Theorem}\label{sec:tcl}

Deriving a central limit theorem for a stochastic algorithm is a common way to assess the asymptotic variance of the algorithm. Central limit theorem for superquantiles has been addressed in \cite{bardou2009computing} for the Cesaro averaging of the sequence $(\vtn)_{n \geq 1}$. Another related relevant work on the subject is the paper of \cite{MokkademPelletier2006} where a general central limit theorem is obtained for stochastic algorithms with two different time scales that allows to study a weak convergence result for the pair $(\tn,\vtn)$ (see also \cite{BCG1}). We also refer to the pioneering works of \cite{borkar,tsitsiklis} for other results on two time-scales algorithms. 
However, in our case, even though  $(\tn,\bvtn)_{n \ge 1}$ seems to evolve with two different time scales $(a_n,b_n)_{n \geq 1}$, the evolution of  $(\bvtn)_{n \geq 1}$ from iteration $n$ to iteration $n+1$ highly depends on $\btn$ and therefore a third time-scale $(1/(n+1))$ is involved. This last remark significantly complicates the use of Theorem 1 of \cite{MokkademPelletier2006}. All the more, it appears that some of the technical assumptions used in \cite{MokkademPelletier2006} (especially Assumption (A4)-$iii)$) can be avoided with the help of another proof strategy.

In what follows, we will first consider the case where $b_n=b_1(n+1)^{-1}$ in which the convergence speed of $(\vtn)_{n \ge 1}$ is optimal, and derive a joint convergence for $(\btn,\bvtn)$ that evolves with a single time scale proportionnal to $1/n$ (see Equation \ref{eq:algo}). To derive a central limit theorem for our pair $(\btn,\bvtn)_{n \geq 1}$, we use the diffusion approximation of stochastic algorithm and we follow the roadmap of \cite{benveniste_metivier_priouret} (see also \cite{Kushner_Yin03}).
We write the joint evolution of a rescaled version of the algorithm $(\btn,\bvtn)_{n \geq 1}$ and introduce:
\begin{equation}\label{eq:algo_rescale}
\Tn = \sqrt{n} (\btn-\ta) \qquad \text{and} \qquad \Vtn = \sqrt{\frac{1}{b_n}} (\bvtn-\vta).% = \sqrt{\frac{n+1}{b_1}} (\bvtn-\vta).
\end{equation}
%Then, we will consider the case $b_n=b_1 n^{-b}$ and derive the diffusion approximation of $\Vtn$.%the rescaled $ (\bvtn-\vta) b_n^{-1/2}$.

\subsection{Case of a step sequence $b_n=b_1/(n+1)$}
\subsubsection{Decomposition of $(\Tn)_{n \geq 1}$ and $(\Vtn)_{n \geq 1}$}

To obtain a central limit theorem, we will use an approximation of the rescaled algorithm by a stochastic differential equation. More precisely, we will show that the pair $(\Tn,\Vtn)_{n \geq 1}$ is close to the discretization of a specific Markov process: the Ornstein-Uhlenbeck diffusion.
Of course, if one considers the first coordinate, we will then recover the central limit theorem for the Cesaro averaging of the quantile sequence.
We emphasize that our method significantly differs from the ones previously developed to establish central limit theorem results for the Ruppert-Polyak averaging algorithm: \cite{polyakjuditsky}, \cite{pelletier2000},\cite{Godichon} are essentially based on the Abel transform and the decomposition of Equation \eqref{eq:decomposition_pelletier}.

\noindent\underline{Recursive analysis of $(\tTn)_{n \geq 1}$.}
Instead of directly studying the pair $(\Tn,\Vtn)_{n \geq 1}$, we first consider the reduction of $(\btn-\ta)$ obtained after our eigenvalue decomposition in Section \ref{sec:averaging_quantile}. We recall the definition of $(\tZn)_{n \geq 1}$ in \eqref{def:tZn} and we introduce:
$$
\tTn := \sqrt{n} \tZn^{(2)} = \sqrt{n} \left[(\btn-\ta)-\epsilon_n(\tn-\ta)\right] =\Tn-\sqrt{n} \epsilon_n (\tn-\ta).
$$
We remark, using \eqref{eq:rec-tZn}, that:
\begin{align*}
\tTnp &= \sqrt{n+1} \tZnp^{(2)} \nonumber \\
& = \sqrt{n+1} \left( \frac{n}{n+1} \tZn^{(2)} + a_n \left( \frac{1}{n+1}-\epsilon_{n+1}\right) \Delta M_{n+1} +\mathcal{R}_n \right) \nonumber \\
& = \sqrt{\frac{n}{n+1}} \tTn + a_n \sqrt{n+1} \left( \frac{1}{n+1}-\epsilon_{n+1}\right) \Delta M_{n+1}   +  \sqrt{n+1}  \mathcal{R}_n.\nonumber 
\end{align*}
Finally
\begin{align}
\tTnp &= \tTn - \frac{1}{2(n+1)} \tTn + a_n \sqrt{n+1} \left( \frac{1}{n+1}-\epsilon_{n+1}\right) \Delta M_{n+1} + \tEnun, \label{eq:linear_tTn}
\end{align}
where $\tEnun$ satisfies:
\begin{equation}\label{eq:borne_tEnun}
|\tEnun| \lesssim \frac{|\tTn|}{n^2} + \frac{|\tn-\ta|}{n^{3/2-a}} + \frac{(\tn-\ta)^2}{\sqrt{n}}.
\end{equation} The previous upper bound is obtained using a Taylor approximation of $\sqrt{n/(n+1)}$ on the first term, the control $\mathcal{R}_n$ defined in \eqref{eq:restes} given by Lemma \ref{lem:t2}, the definition of $(\epsilon_n)_{n \geq 1}$ and Equation \eqref{eq:reste_1}.

\noindent\underline{Recursive analysis of $(\Vtn)_{n \geq 1}$.}
We use the recursive formulation given in Proposition \ref{prop:linearization}:
\begin{align*}
  \Vtnp &=\Vtn \sqrt{1+\frac{1}{n+1}} \left(1-\frac{b_1}{n+1}\right) - \frac{\sqrt{b_1}\ta f(\ta)}{1-\alpha} \frac{\sqrt{n+2}}{\sqrt{n} (n+1)} \Tn  \\
 & + \frac{\sqrt{b_1}}{(1-\alpha)} \frac{\sqrt{n+2}}{n+1} \Delta N_{n+1} \nonumber  +\frac{ \sqrt{b_1}}{1-\alpha} \frac{\sqrt{n+2}}{n+1} Q''(\xi_{n+1}) (\btn-\ta)^2.
 \end{align*}
 Using $\Tn=\tTn+\epsilon_n\sqrt{n}(\tn-\ta)$ and  a first order Taylor expansion of $\sqrt{1+1/(n+1)}$ and of $\sqrt{(n+2)/n}$, we deduce that:
 \begin{align}
  \Vtnp & = 
  \Vtn - \frac{b_1-\frac{1}{2}}{n+1} \Vtn  - \frac{ \frac{\sqrt{b_1}\ta f(\ta)}{1-\alpha}}{n+1} \tTn +\frac{ \sqrt{b_1}}{(1-\alpha)}\frac{\sqrt{n+2}}{n+1} \Delta N_{n+1} +\tEndeux,
  \label{eq:linear_tVn}
  \end{align}
%- \frac{\frac{\sqrt{b_1}f(\ta)}{1-\alpha}}{n} \Ttn
where $\tEndeux$ satisfies:
\begin{equation}
\label{eq:borne_tEndeux}
|\tEndeux| \lesssim \frac{|\Vtn|}{n^2} +
\frac{|\tTn|}{n^2} + \frac{|\btn-\ta|^2}{\sqrt{n}} + \frac{|\tn-\ta|}{n^{3/2-a}}.
\end{equation}

Gathering Equations \eqref{eq:linear_tTn} and \eqref{eq:linear_tVn}, we have obtained the following proposition.
\begin{prop}
\label{prop:lin_TCL}
For any $n \geq 1$, define $\Tzn = (\tTn,\Vtn)$, then:
$$
\Tznp= \Tzn - \underbrace{\frac{1}{n+1}
 \left( \begin{matrix}
 1/2 & 0 \\
 \frac{\sqrt{b_1} \ta f(\ta)}{1-\alpha} & b_1-1/2
\end{matrix} \right)\Tzn + \tEn }_{:=H_n(\Tzn)} + \frac{1}{\sqrt{n+1}} 
\underbrace{\left(
\begin{matrix}
 a_n(n+1) \left(\frac{1}{n+1}-\epsilon_{n+1}\right) \Delta M_{n+1}\\
\frac{\sqrt{b_1}}{ (1-\alpha)}\sqrt{\frac{n+2}{n+1}} \Delta N_{n+1} 
\end{matrix}
\right)}_{:=\Delta\mathcal{M}_{n+1}}, 
$$
where $\tEn=(\tEnun,\tEndeux)$ satisfies:
\begin{align}
\|\tEn\| & \lesssim  \frac{|\btn-\ta|}{n^{3/2}}  +\frac{|\btn-\ta|^2}{\sqrt{n}}  
+
 \frac{|\tn-\ta|}{n^{3/2-a}} + \frac{|\tn-\ta|^2}{\sqrt{n}}+\frac{|\bvtn-\vta|}{n^{3/2}}.
\label{eq:reste}
\end{align}
\end{prop}

\subsubsection{Tightness and limit of the martingale increments}

We use some standard notations of diffusion approximations of rescaled stochastic algorithms. We introduce:
$$
\Gamma_n = \sum_{k=1}^n \frac{1}{k},
$$
The number of iterations needed to last a time $t$ after a shift of $\Gamma_n$ is then denoted by
$\underline{t}_n$:
$$\underline{t}_n= \Gamma_{N(n,t)}- \Gamma_n \qquad \text{and} \qquad N(n,t) = \max\left\{m\geq n,  \sum\limits_{k=n+1}^m \frac{1}{k}<t\right\}.$$
We associate to the sequence $(\Tzn)_{n\geq 1}$ a sequence $(\Tztn)_{n\ge1}$ of time-shifted continuous processes defined as follows: for any integer $n$, the process $(\Tztn_t)_{t \geq 0}$ corresponds to a continuous-time affine interpolation of the rescaled stochastic algorithm that starts at position $\Tzn$ at time $0$:
$$\Tztn_t=\tilde{\mathbb{Z}}_{N(n,t)}+(N(n,t)+1)(t-\underline{t}_n)(\tilde{\mathbb{Z}}_{N(n,t)+1}-\tilde{\mathbb{Z}}_{N(n,t)}).$$
Using the notations of Proposition \ref{prop:lin_TCL} we remark  that:
\begin{equation}\label{eq:Tztn}
\forall t \geq 0 \qquad \Tztn_t= \Tzn+ B_t^{(n)}+M_t^{(n)},
\end{equation}
where: 
$$B_t^{(n)} = - \sum\limits_{k=n+1}^{N(n,t)}  H_{k-1}(\mathbb{Z}_{k-1}) - (t-\underline{t}_n) (N(n,t)+1)H_{N(n,t)}(\mathbb{Z}_{N(n,t)}),$$
and:
$$
M_t^{(n)}=\sum\limits_{k=n+1}^{N(n,t)}\sqrt{ \frac{1}{k}}\Delta\mathcal{M}_k+\sqrt{N(n,t)+1}(t-\underline{t}_n)\Delta\mathcal{M}_{N(n,t)+1}.$$

In this paragraph, we identify the possible weak limits of the sequence $\Tztn$ (see \textit{e.g.} \cite{benveniste_metivier_priouret}) and some  details are skipped for the sake of convenience. The reader may found similar arguments in Section 5 of \cite{gadat_panloup_saadane}.
%The sequence $\Tztn$  may be approximated by an Ornstein-Ulhenbeck process and we will characterize in the sequel the drift and diffusion coefficients.

We define the infinitesimal generator $\mathcal{G}$ : for any $\varphi \in \mathcal{C}_b^2(\mathbb{R}^2,\mathbb{R})$:
\begin{align}
\label{eq:generateur}
\mathcal{G}(\varphi)(x,y)=&-\frac{x}{2} \partial_x \varphi - \left(\frac{\sqrt{b_1} \ta f(\ta)}{1-\alpha} x + (b_1-1/2) y\right) \partial_y \varphi \nonumber\\
&+ \frac{\alpha(1-\alpha)}{2 f(\ta)^2} \partial_{xx}^2 \varphi + \frac{b_1 V_{\alpha}}{2(1-\alpha)^2} \partial_{yy}^2 \varphi + \sqrt{b_1}\frac{\alpha \vta}{f(\ta)} \partial_{xy}^2 \varphi.
\end{align}

\begin{prop}\label{prop:tightness}
\begin{itemize} 
% The following assertions hold.
\item[$i)$]The sequence $(\Tztn)_{n \geq 1}$ is tight for the weak topology induced by the weak convergence on compact intervals of continuous time processes.
\item[$ii)$]If $V_\alpha$ is defined in \eqref{def:Valpha} then
$$\E[\Delta\mathcal{M}_{n+1}\Delta\mathcal{M}_{n+1}^t \lvert \mathcal{F}_{n}]=\begin{pmatrix}
\frac{\alpha(1-\alpha)}{f(\ta)^2} &\frac{\sqrt{b_1}\alpha\vta}{f(\ta)}\\
\frac{\sqrt{b_1}\alpha\vta}{f(\ta)}&\frac{b_1 V_\alpha}{(1-\alpha)^2}
\end{pmatrix}+\upsilon_n,
$$
where  $\upsilon_n \longrightarrow 0$ in $L^1$.
\item[$iii)$] For any function $\varphi \in \mathcal{C}^2_b(\mathbb{R}^2,\mathbb{R})$, 
$$
\E \left[\varphi(\Tznp) \, \vert \Tzn \right] = \varphi (\Tzn) + \frac{1}{n+1} \mathcal{G}(\varphi)(\Tzn)+\mathbb{Q}_n,
$$
where $\lim_{n \longrightarrow + \infty} n \E |\mathbb{Q}_n| = 0$.
\end{itemize}
\end{prop}

The proof of this previous   compactness result is deferred to Appendix \ref{app:tcl}.

\subsubsection{Central Limit Theorem - Theorem \ref{theo:tcl}-$ii)$}

\begin{proof}[Proof of Theorem \ref{theo:tcl}-$ii)$]
The proof is briefly sketched since it exactly follows the same lines as those of Theorem 2.4 of \cite{gadat_panloup_saadane} (Section 5.3).\\
\noindent
$\bullet$
From Proposition \ref{prop:tightness}-$i)$, we know  that the sequence of \textit{processes} $(\Tztn)_{n \geq 1}$ is tight and Proposition \ref{prop:tightness}-$ii)/iii)$ entails that any weak limit of $(\Tztn)_{n \geq 1}$ is solution of the martingale problem with generator $\mathcal{G}$ on the domain of twice differentiable functions $\mathcal{C}^2_b(\mathbb{R}^2,\mathbb{R})$. We emphasize that $\mathcal{G}$ corresponds to an Ornstein-Uhlenbeck Markov generator. The associated semi-group $P_t = e^{t \mathcal{G}}$ is elliptic and ergodic: a unique invariant Gaussian distribution $\nu_{\infty}$ exists such that for any test function $\varphi$: $P_t(\varphi) \longrightarrow \nu_{\infty}(\varphi)$ as $t \longrightarrow + \infty$
uniformly on compact sets of $\R^2$. We compute the limit covariance in Proposition \ref{prop:tightness}-$ii)$.

\noindent
$\bullet$
Second, the normalized algorithm $(\Tzn)_{n \geq 1}$ is also a tight sequence: we consider a possible weak limit $\mu$ associated to an extracted sequence $(\tilde{\mathbb{Z}}_{n_k})_{k \geq 1}$ and a time $T$ large enough. Then, we prove that an associated subseqence of shifted processes $\tilde{\mathbb{Z}}^{(m_k)}$ exists such that $\tilde{\mathbb{Z}}^{(m_k)}_T$ is arbitrarily close to $\nu_{\infty}$ and to $\mu$, because of the shifted nature of continuous time processes.

\noindent
$\bullet$ We then conclude that $\mu = \nu_{\infty}$, meaning that $(\Tzn)_{n \geq 1}$ is a tight sequence with a unique adherence point. It proves that $ \Tzn   \underset{n \to +\infty}{\overset{\mathcal{L}}{\longrightarrow}} \nu_{\infty}$.
We recall that $\nu_{\infty}$ is a bi-dimensional gaussian with mean $0$ and covariance matrix $\Sigma$.\\
We finally obtain the convergence of  $\sqrt{n} (\btn-\ta,\bvtn-\vta)$  using the Slutsky theorem. Let us remark that:
$$\sqrt{n} \begin{pmatrix}\btn-\ta\\
\bvtn-\vta\end{pmatrix} = \begin{pmatrix}
\Tzn^{(1)}+\sqrt{n}\epsilon_n(\tn-\ta)\\
\sqrt{b_1}\sqrt{\frac{n}{n+1}}\Tzn^{(2)}
\end{pmatrix}.
$$
We verify that:
$$
\E  \sqrt{n} \epsilon_n |\tn-\ta| \lesssim \frac{1}{\sqrt{n a_n}} \longrightarrow 0 \quad \text{as} \quad n \longrightarrow + \infty,
$$
so that $\sqrt{n} \epsilon_n (\tn-\ta)$ converges to $0$ in probability (even in $L^1$).
We therefore deduce that $\sqrt{n} (\btn-\ta,\bvtn-\vta)$  converges  to a Gaussian distribuion with mean $0$ and covariance matrix
$$
S^2 =\begin{pmatrix}
\Sigma_{xx} & \sqrt{b_1} \Sigma_{xy} \\
\sqrt{b_1} \Sigma_{xy} &  b_1  \Sigma_{yy}
\end{pmatrix}.
$$
\noindent
$\bullet$
We finally detail the computation of the limiting variance $\Sigma$ and  identify the invariant measure $\nu_{\infty}$ associated to the infinitesimal generator $\mathcal{G}$ defined in \eqref{eq:generateur}.
\begin{itemize}
\item[$\diamond$] We apply $\int \mathcal{G}\varphi(x,y) \nu_{\infty}(dx,dy)=0$ to $\varphi(x,y)=x^2/2$ and observe that:
$$\int \left[-\frac{x^2}{2} +\frac{\alpha(1-\alpha)}{2f(\ta)^2} \right] \nu_{\infty}(dx,dy) =0,$$
which leads to:
$$\Sigma_{xx}=\frac{\alpha(1-\alpha)}{f(\ta)^2}.$$
\item[$\diamond$] Then using the function $\varphi(x,y)=xy$, we obtain that:
$$\int -\frac{xy}{2} -\left(\frac{\sqrt{b_1} \ta f(\ta)}{1-\alpha} x + (b_1-1/2) y\right)x +\frac{\sqrt{b_1}\alpha \vta}{f(\ta)} \nu_{\infty}(dx,dy)=0.$$
Therefore, we deduce that:
$$\Sigma_{xy}b_1 = \frac{\sqrt{b_1}\alpha \vta}{f(\ta)}-\frac{\sqrt{b_1} \ta f(\ta)}{1-\alpha} \Sigma_{xx},$$
and 
$$\Sigma_{xy}=\frac{\alpha}{f(\ta)\sqrt{b_1}}(\vta-\ta).$$
\item[$\diamond$]
Finally using $\varphi(x,y)=y^2/2$, we have
$$\int - \left(\frac{\sqrt{b_1} \ta f(\ta)}{1-\alpha} x + (b_1-1/2) y\right) y +\frac{b_1 V_{\alpha}}{2(1-\alpha)^2} \nu_{\infty}(dx,dy)=0,$$
and thus
$$\Sigma_{yy}=\frac{2}{2b_1 -1}\left( b_1\frac{V_\alpha}{2(1-\alpha)^2}-\frac{\alpha\ta}{(1-\alpha)}(\vta-\ta)\right).$$
\end{itemize}
\end{proof}

\subsection{Case of a step sequence $b_n=b_1/n^b$}
Since the computations are rather similar, we only sketch the essential steps of the proof.

\begin{proof}[Proof of Theorem \ref{theo:tcl}-$i)$]
We still use Proposition \ref{prop:linearization} as:
\begin{align*}
  \Vtnp &=\Vtn \sqrt{\frac{b_n}{b_{n+1}}} \left(1-b_n\right)+\sqrt{b_n}\sqrt{\frac{b_n}{b_{n+1}}}\frac{ \Delta N_{n+1}}{1-\alpha}  \\
 & -\frac{b_n}{\sqrt{b_{n+1}}}(\btn-\ta)\frac{\ta f(\ta)}{1-\alpha}+\frac{ b_n}{\sqrt{b_{n+1}}(1-\alpha)} Q''(\xi_{n+1}) (\btn-\ta)^2.
 \end{align*}
Following Proposition \ref{prop:lin_TCL}, we can identify the main part of the recursion and the  rest terms:
\begin{align*}
  \Vtnp &=\Vtn -\Vtn b_n+ \mathcal{E}_{n+1}+\sqrt{b_n}\frac{ \Delta N_{n+1}}{1-\alpha}  \sqrt{\frac{b_n}{b_{n+1}}},
 \end{align*}
 where 
 $$\mathcal{E}_{n+1}= \Vtn (1-b_n)\left(\sqrt{\frac{b_n}{b_{n+1}}}-1\right)  -\frac{b_n}{\sqrt{b_{n+1}}}(\btn-\ta)\frac{\ta f(\ta)}{1-\alpha}+\frac{ b_n}{\sqrt{b_{n+1}}(1-\alpha)} Q''(\xi_{n+1}) (\btn-\ta)^2.$$
Let us remark that:
$$\lvert\mathcal{E}_{n+1}\lvert \lesssim \frac{\lvert \Vtn\lvert}{n} + n^{-b/2} \lvert \btn-\ta\lvert+n^{-b/2} \lvert \btn-\ta\lvert^2.$$
Then, the arguments developed in the previous sections can be easily adapted. In the following we only detail the modifications and skip below the technical details related to the discrete/continuous time-scale.

\noindent
\underline{Tightness of $(\Vtn)_{n \ge 1}$.}
The proof is similar to the one stated in Appendix \ref{app:tcl}. Indeed, $\sup_{n \geq 1} \mathbb{E} \|\Vtn \|^2 < + \infty$ and for any $\delta>0$, $\sum_{N(n,t)}^{N(n,t+\delta)+1}   b_k  \leq 2 \delta$ for large $n$, thus 
\begin{equation}\label{eq:tension1_bis}
\mathbb{P}\left( \left| \sum_{N(n,t)}^{N(n,t+\delta)+1}  b_k\mathbb{V}_{k} \right| \geq \epsilon \right) \lesssim
\epsilon^{-2} \delta^2.
\end{equation}
The only modification comes from the rest terms:
\begin{align*}
\left\|\sum_{N(n,t)}^{N(n,t+\delta)+1}  \mathcal{E}_k \right\|^2& \lesssim \left(  \sum_{N(n,t)}^{N(n,t+\delta)+1} 
\frac{|\bar{\theta}_k-\ta|}{k^{b/2}}+
\frac{|\bar{\theta}_k-\ta|^2}{k^{b/2}} + 
\frac{|\widehat{\vartheta}_k-\vta|}{k^{1-b/2}} 
  \right)^2 \\
  & \lesssim  \left(\sum_{N(n,t)}^{N(n,t+\delta)} 
\frac{|\bar{\theta}_k-\ta|}{k^{b/2}}\right)^2 + 
\left( \sum_{N(n,t)}^{N(n,t+\delta)} 
\frac{|\bar{\theta}_k-\ta|^2}{k^{b/2}} \right)^2 
+\left( \sum_{N(n,t)}^{N(n,t+\delta)} 
\frac{|\widehat{\vartheta}_k-\vta|}{k^{1-b/2}} \right)^2.
\end{align*}
We apply the Cauchy-Schwarz inequality:
\begin{align*}
\left\|\sum_{N(n,t)}^{N(n,t+\delta)+1} \mathcal{E}_k \right\|^2&\lesssim \left(\sum_{N(n,t)}^{N(n,t+\delta)+1} 
b_k\right)\times \left[
\sum_{N(n,t)}^{N(n,t+\delta)+1} 
|\bar{\theta}_k-\ta|^2+ 
|\bar{\theta}_k-\ta|^4   +\frac{|\widehat{\vartheta}_k-\vta|^2}{k^{2-2b}} \right].
\end{align*}
We shall now compute the expectation of the previous terms and verify that: 
\begin{align*}
\E \left\|\sum_{N(n,t)}^{N(n,t+\delta)}  \tilde{\mathcal{E}}_k \right\|^2 &\lesssim 
\left(\sum_{N(n,t)}^{N(n,t+\delta)} 
b_k
  \right) \left(\sum_{N(n,t)}^{N(n,t+\delta)} 
b_k \left[ \frac{\E  |\bar{\theta}_k-\ta|^2}{b_k}
+ \frac{\E  |\bar{\theta}_k-\ta|^4}{b_k} +\frac{\E  |\widehat{\vartheta}_k-\vta|^2}{k^{2-b}}
\right]
  \right).
\end{align*}
Using now  Theorem \ref{theo:rate_averaging} ($L^p$ loss on the sequence $(\btn)_{n \geq 1}$) and Theorem \ref{theo:rate_vtn} ($L^2$ loss on the sequence $(\vtn)_{n \geq 1}$),
we obtain that the bracket term is uniformly bounded in $k$, and therefore:
$$
\E \left\|\sum_{N(n,t)}^{N(n,t+\delta)}  \tilde{\mathcal{E}}_k \right\|^2 \lesssim \left(\sum_{N(n,t)}^{N(n,t+\delta)} 
b_k
  \right)^2 \lesssim \delta^2.
$$
Hence, the sequence $\Vtn$ is tight. Finally, applying the above arguments and the bound on $\mathcal{E}_n$, any weak limit is solution of the martingale problem associated with the generator: 
\begin{align}
\label{eq:generateur_bis}
\mathcal{G}(\varphi)(x)=&-x\varphi'(x)  + \frac{ V_{\alpha}}{2(1-\alpha)^2}  \varphi''(x).
\end{align}
%\no{En fait je pense qu'il n'y a pas le terme $b_1$. Le générateur limite est juste 
%$$\mathcal{G}(\varphi)(x)=-x\varphi'(x)  +  \frac{ V_{\alpha}}{2(1-\alpha)^2}  \varphi''(x).$$}
It remains to identify the invariant measure associated with the Ornstein-Ulhenbeck diffusion \eqref{eq:generateur_bis}. 
An easy computation with $\varphi(x)=x^2/2$ yields the next result.
$$
\frac{\bvtn-\vta}{\sqrt{b_n}} \underset{n \to +\infty}{\overset{\mathcal{L}}{\longrightarrow}} \mathcal{N}\left(0,\frac{ V_{\alpha}}{2(1-\alpha)^2}\right).
$$
\end{proof}

\section*{Acknowledgments}
The authors would like to warmly thank Bernard Bercu for fruitful discussions about their work. They also thank the two anonymous referees for their insightful comments and
constructive suggestions which helped to improve the paper.
%
%\bibliographystyle{alpha}
%\bibliography{biblio}  

\begin{thebibliography}{LRGGI16}

\bibitem[ADEH97]{Artzner97}
P.~Artzner, F.~Delbaen, J.M. Eber, and D.~Heath.
\newblock Thinking coherently.
\newblock {\em Risk}, 10, 1997.

\bibitem[ADEH99]{axiomatique}
P.~Artzner, F.~Delbaen, J.M. Eber, and D.~Heath.
\newblock Coherent measures of risk.
\newblock {\em Mathematical Finance}, 9(3), 1999.

\bibitem[BBH09]{Balbas}
A.~Balbas, B.~Balbas, and A.~Heras.
\newblock Optimal reinsurance with general risk measures involving risk
  measures.
\newblock {\em European Journal of Operational Research}, pages 1--29, 2009.

\bibitem[BCG21]{BCG1}
B.~Bercu, M.~Costa, and S.~Gadat.
\newblock Stochastic estimation algorithm for superquantile estimation.
\newblock {\em Electronic Journal of Probability, to appear}, 2021.

\bibitem[Ben99]{benaim1999dynamics}
M.~Bena{\"\i}m.
\newblock Dynamics of stochastic approximation algorithms.
\newblock In {\em S{\'e}minaire de probabilit{\'e}s XXXIII}, pages 1--68.
  Springer, 1999.

\bibitem[BFP09]{bardou2009computing}
O.~Bardou, N.~Frikha, and G.~Pages.
\newblock Computing var and cvar using stochastic approximation and adaptive
  unconstrained importance sampling.
\newblock {\em Monte Carlo Methods and Applications}, 15(3):173--210, 2009.

\bibitem[Bil95]{billingsley}
P.~Billingsley.
\newblock {\em Convergence of Probability Measures}.
\newblock Wiley series in Probability \& Statistics, New York, 1995.

\bibitem[BJS99]{NonlinearVar1}
M.~Britten-Jones and S.M. Schaefer.
\newblock Non linear value-at-risk.
\newblock {\em European Finance Review}, 2, 1999.

\bibitem[BM11]{Bach_Moulines_2011}
F.~Bach and E.~Moulines.
\newblock {Non-Asymptotic Analysis of Stochastic Approximation Algorithms for
  Machine Learning}.
\newblock In {\em {NIPS}}, pages 451--459, 2011.

\bibitem[BMP90]{benveniste_metivier_priouret}
A.~Benveniste, M.~M{\'e}tivier, and P.~Priouret.
\newblock {\em Adaptive Algorithms and Stochastic Approximations}.
\newblock Springer-Verlag, Berlin, New-York, Stochastic Modelling and Applied
  Probability, 1990.

\bibitem[Bor97]{borkar}
V.S. Borkar.
\newblock Stochastic approximation with two time scales.
\newblock {\em Systems Control Lett.}, 29, 1997.

\bibitem[BTT86]{Ben-Tal}
A.~Ben-Tal and M.~Teboulle.
\newblock Expected utility, penalty functions, and duality in stochastic
  nonlinear programming.
\newblock {\em Management Science}, 32(11):1445--1466, 1986.

\bibitem[BZ10]{Beutner}
E.~Beutner and H.~Zähle.
\newblock A modified functional delta method and its application to the
  estimation of risk functionals.
\newblock {\em Journal of Multivariate Analysis}, 101(10), 2010.

\bibitem[CCGB17]{Godichon}
H.~Cardot, P.~C\'enac, and A.~Godichon-Baggioni.
\newblock Online estimation of the geometric median in {H}ilbert spaces:
  {N}onasymptotic confidence balls.
\newblock {\em The Annals of Statistics}, 45(2):591--614, 2017.

\bibitem[CCZ13]{Cenac}
H.~Cardot, P.~Cenac, and P.A. Zitt.
\newblock Efficient and fast estimation of the geometric median in hilbert
  spaces with an averaged stochastic gradient algorithm.
\newblock {\em Bernoulli}, 19:18--43, 2013.

\bibitem[CU01]{Chernozhukov}
V.~Chernozhukov and L.~Umantsev.
\newblock Conditional value-at-risk: Aspects of modeling and esti- mation.
\newblock {\em Empirical Economics}, 26(1), 2001.

\bibitem[Del09]{Delbaen}
F.~Delbaen.
\newblock Risk measures for non–integrable random variables.
\newblock {\em Mathematical Finance}, (19):329--333, 2009.

\bibitem[DHF06]{EVT}
L.~De~Haan and A.~Ferreira.
\newblock {\em Extreme Value Theory, An Introduction}.
\newblock Springer Series in Operations Research and Financial Engineering,
  2006.

\bibitem[DP01]{NonlinearVar2}
D.~Duffie and J.~Pan.
\newblock Analytical value-at-risk with jumps and credit risk.
\newblock {\em Finance and Stochastics}, 5, 2001.

\bibitem[Duf97]{Duflo97}
M.~Duflo.
\newblock {\em {Random iterative models}}, volume~34 of {\em {Applications of
  Mathematics (New York)}}.
\newblock Springer-Verlag, Berlin, 1997.

\bibitem[Emb99]{Embrechts99}
P.~Embrechts.
\newblock Extreme value theory as a risk management tool.
\newblock {\em North American Actuarial Journal}, 3(2), 1999.

\bibitem[ENW09]{Embrechts}
P.~Embrechts, J.~Neslehova, and M.~Wuthrich.
\newblock Additivity properties for value-at-risk under archimedean dependence
  and heavy-tailedness.
\newblock {\em Insurance: Mathematics and Economics}, 44:164--169, 2009.

\bibitem[God15]{godichon2015}
A.~Godichon.
\newblock Estimating the geometric median in {H}ilbert spaces with stochastic
  gradient algorithms : Lp and almost sure rates of convergence.
\newblock {\em Journal of Multivariate Analysis}, pages 209--222, 2015.

\bibitem[GP20]{Gadat-Panloup}
S.~Gadat and F.~Panloup.
\newblock {Optimal non-asymptotic bound of the Ruppert-Polyak averaging without
  strong convexity}.
\newblock 2020.

\bibitem[GPS18]{gadat_panloup_saadane}
S.~Gadat, F.~Panloup, and S.~Saadane.
\newblock Stochastic heavy ball.
\newblock {\em Electronic Journal of Statistics}, pages 461--529, 2018.

\bibitem[HY03]{Hall-Yao}
P.~Hall and Q.~Yao.
\newblock Inference in {ARCH} and {GARCH} models with heavy-tailed errors.
\newblock {\em Econometrica}, 71, 2003.

\bibitem[KT04]{tsitsiklis}
V.R. Konda and J.N. Tsitsiklis.
\newblock Convergence rate of linear two-time-scale stochastic approximation.
\newblock {\em The Annals of Applied Probability}, 14, 2004.

\bibitem[KY03]{Kushner_Yin03}
H.~J. Kushner and G.~Yin.
\newblock {\em {Stochastic approximation and recursive algorithms and
  applications}}, volume~35.
\newblock Springer Verlag, 2003.

\bibitem[LRGGI16]{Garivier2}
T.~Labopin-Richard, F.~Gamboa, A.~Garivier, and B.~Iooss.
\newblock Bregman superquantiles. {E}stimation methods and applications.
\newblock {\em Dependence Modeling}, 4(1), 2016.

\bibitem[Man63]{Mandelbrot}
B.~Mandelbrot.
\newblock The variation of certain speculative prices.
\newblock {\em Journal of Business}, (26):394--419, 1963.

\bibitem[MP87]{metivier1987theoremes}
M.~M{\'e}tivier and P.~Priouret.
\newblock Th{\'e}or{\`e}mes de convergence presque sure pour une classe
  d'algorithmes stochastiques {\`a} pas d{\'e}croissant.
\newblock {\em Probability Theory and related fields}, 74(3):403--428, 1987.

\bibitem[MP06]{MokkademPelletier2006}
A.~Mokkadem and M.~Pelletier.
\newblock Convergence rate and averaging of nonlinear two-time-scale stochastic
  approximation algorithms.
\newblock {\em The Annals of Applied Probability}, 16, 11 2006.

\bibitem[Pel00]{pelletier2000}
Mariane Pelletier.
\newblock Asymptotic almost sure efficiency of averaged stochastic algorithms.
\newblock {\em SIAM, Journal on Control and Optimization}, (1):49--72, 2000.

\bibitem[Pfl00]{Pflug00}
G.~Pflug.
\newblock Some remarks on the value-at-risk and the conditional value-at-risk.
\newblock {\em In.”Probabilistic Constrained Optimization: Methodology and
  Applications”, Ed. S. Uryasev, Kluwer Academic Publishers}, 2000.

\bibitem[PJ92]{polyakjuditsky}
B.~T. Polyak and A.~Juditsky.
\newblock Acceleration of stochastic approximation by averaging.
\newblock {\em SIAM Journal on Control and Optimization}, 30:838--855, 1992.

\bibitem[RM00]{Rachev}
S.~T. Rachev and S.~Mittnik.
\newblock {\em Stable Paretian Models in Finance}.
\newblock John Wiley and Sons, Series in Financial Economics., 2000.

\bibitem[Rou97]{NonlinearVar3}
C.~Rouvinez.
\newblock Going {G}reek with {V}a{R}.
\newblock {\em Risk}, 10, 1997.

\bibitem[RU00]{Cvar}
R.~T. Rockafellar and S.~Uryasev.
\newblock Optimization of conditional value-at-risk.
\newblock {\em The Journal of Risk}, 2(3), 2000.

\bibitem[Rup88]{Ruppert}
D.~Ruppert.
\newblock Efficient estimations from a slowly convergent robbins-monro process.
\newblock {\em Technical Report, 781, Cornell University Operations Research
  and Industrial Engineering}, 1988.

\bibitem[TPFG20]{Garivier1}
L.~Torossian, V.~Picheny, R.~Faivre, and A.~Garivier.
\newblock A review on quantile regression for stochastic computer experiments.
\newblock {\em Reliability Engineering \& System Safety}, 201, 2020.

\bibitem[WZ16]{Wang-Zhao}
C.-S. Wang and Z.~Zhao.
\newblock Conditional value-at-risk: Semiparametric estimation and inference.
\newblock {\em Journal of Econometrics}, 195:86--103, 2016.

\end{thebibliography}
\providecommand{\AC}{A.-C}\providecommand{\CA}{C.-A}\providecommand{\CH}{C.-H}\providecommand{\CJ}{C.-J}\providecommand{\JC}{J.-C}\providecommand{\JP}{J.-P}\providecommand{\JB}{J.-B}\providecommand{\JF}{J.-F}\providecommand{\JJ}{J.-J}\providecommand{\JM}{J.-M}\providecommand{\KW}{K.-W}\providecommand{\PL}{P.-L}\providecommand{\RE}{R.-E}\providecommand{\SJ}{S.-J}\providecommand{\XR}{X.-R}\providecommand{\WX}{W.-X}

\newpage

\appendix
\section{From Lyapunov contraction to non-asymptotic upper bound}

The next lemmas provide non-asymptotic upper-bounds from a Lyapunov-type contraction. 
\begin{lem}\label{lem:recurrence_finale_2}
Let $b_n=b_1n^{-b}$ with $\frac{1}{2}<b<1$, $A,C>0$, $r_1>\frac{3b}{2}$ and $r_2>2b$.
Consider a positive sequence $u_n$ such that:
$$
u_{n+1} \leq \left(1-b_n\right)^2 u_n + Cb_n^2 + A \left[n^{-r_1} \sqrt{u_n} + n^{-r_2}\right], 
$$
then, for any $p \in ]b,(r_1-b/2) \wedge (r_2-b )\wedge 2b\wedge 1[$, a large enough $\Gamma_p$ exists such that:
\begin{equation}
\label{eq:recursion} 
u_n \leq \frac{C}{2 }b_n + \Gamma n^{-p},
\end{equation}
%where $$b< p \leq (r_1-b/2) \wedge (r_2-b )\wedge 2b\wedge 1.$$
\end{lem}

\begin{proof}
We prove the result with a recursive argument. 
The initialization is obtained for a large enough $\Gamma$. Assume that Equation \eqref{eq:recursion} holds true for an integer $n$, then we have 
\begin{align*}
u_{n+1}& \leq \frac{C}{2 }b_n \left(1-b_n\right)^2 
+ \Gamma n^{-p} \left(1-b_n\right)^2+\frac{C}{(n+1)^2} \\
&\qquad \qquad \qquad + A \left[n^{-r_1} \sqrt{\frac{C}{2}b_n+\Gamma n^{-p}}+n^{-r_2}\right] \\
& \leq \frac{C}{2}b_{n+1}+\Gamma (n+1)^{-p} + e_{n+1}+\tilde{e}_{n+1},
\end{align*}
where: $$
e_{n+1}=\frac{C}{2}b_n\left(1-b_n\right)^2  - \frac{C}{2}b_{n+1} +Cb_n^2,
$$
and:
\begin{align*}
\tilde{e}_{n+1}& =  \Gamma n^{-p} \left(1-b_n\right)^2 - \Gamma (n+1)^{-p} +A n^{-r_2}+A n^{-r_1}\sqrt{\frac{C}{2}b_n+\Gamma n^{-p}}.
\end{align*} 
The aim is now to prove that for a good choice of $p$, $e_{n+1}+\tilde{e}_{n+1}$ is negative.
For $e_{n+1}$ an easy computation leads to:
$$
e_{n+1} =\frac{C}{2}(b_n-b_{n+1}+b_n^3)=O(n^{-3b}\vee n^{-b-1}).
$$
For  $\tilde{e}_{n+1}$, we combine the two first terms
to obtain:
$$\Gamma(n+1)^{-p}\left[\left(1+\frac{1}{n}\right)^p\left(1-b_n\right)^2-1\right]= \Gamma(n+1)^{-p}\left[-2b_n +O(n^{-1})\right].$$
This difference  is therefore negative as soon as $n$ is large enough. Finally:
\begin{align*}
\tilde{e}_{n+1}&= \Gamma (n+1)^{-(p+b)}\left[-2b_1 +\frac{A}{\Gamma} n^{p+b-r_2}+n^{p+b-(r_1+b/2)}\frac{A}{\Gamma} \sqrt{\frac{C}{2 }}\right.\\
& \qquad+ \left.
O(n^{-r_1+b)})+O(n^{-1+b})\right].
\end{align*}

Now, we should remark that if
$p +b \leq (r_1+b/2) \wedge r_2 \wedge 3b\wedge (b+1)$, then a large enough $\Gamma$ exists such that $e_{n+1}+\tilde{e}_{n+1} \leq 0$.
It concludes the proof.
\end{proof}

We now consider the limit case where $r_1=3/2$. 
\begin{lem}\label{lem:recurrence_finale_3}
Let $\gamma>1/2$, $A_1,A_2,C>0$, $r_1>3/2$ and $r_2>2$ and consider a positive sequence $u_n$ such that
$$
u_{n+1} \leq \left(1-\frac{\gamma}{n+1}\right)^2 u_n + \frac{C}{(n+1)^2} +  \sqrt{u_n}\left(A_1 n^{-3/2}+A_2n^{-r_1}\right)+ A_3n^{-r_2}, 
$$
Let us define:
 $$
V= \left(\frac{A_1}{2(2\gamma-1)} + \sqrt{\frac{A_1^2}{4(2\gamma-1)^2}+\frac{C}{2\gamma-1}}\right)^2 \quad \text{and}  \quad \rho=(r_2-1) \wedge (r_1-\frac{1}{2})\wedge 2 >1.
$$
If $\rho-2\gamma+A_1/(2\sqrt{V})<0$ then, a large enough $\Gamma$ exists such that:
\begin{equation*}%\label{eq:recursion}
\forall n \geq 1 \qquad 
u_n \leq \frac{V}{n}+\Gamma n^{-\rho}.
\end{equation*}
\end{lem}
\begin{rmq}
Using the value of $V$, we have $\sqrt{V}\ge 2\frac{A_1}{2(2\gamma-1)} $ thus
$$\frac{A_1}{2\sqrt{V}} \leq \gamma - \frac{1}{2}.$$
Therefore the condition $\rho-2\gamma+A_1/(2\sqrt{V})<0$ is automatically  satisfied as soon as $\gamma>\rho-1/2$.
\end{rmq}

\begin{proof}
We proceed by induction on $n$, in order to prove that 
$
u_n \leq \frac{V}{n} + \Gamma n^{-\rho},
$
for the largest possible value of $\rho$.\\
The initialization is obvious at the price of a large enough $\Gamma$. We consider the recursion property assuming that the property holds at iteration $n$. Then:
\begin{align*}
u_{n+1} &\leq \left(1-\frac{\gamma}{n+1}\right)^2 u_n + \frac{C}{(n+1)^2} + (A_1 n^{-3/2}+A_2n^{-r_1}) \sqrt{u_n} + A_3n^{-r_2} \\
& \leq \left(1-\frac{\gamma}{n+1}\right)^2 \left[  \frac{V}{n} + \Gamma n^{-\rho}\right] + \frac{C}{(n+1)^2} + (A_1 n^{-3/2} +A_2n^{-r_1}) \sqrt{\frac{V}{n} + \Gamma n^{-\rho}} + A_3 n^{-r_2} \\
& = \frac{V}{n+1} + \Gamma (n+1)^{-\rho} + \left[\frac{V}{n}-\frac{V}{n+1} - \frac{2 \gamma V}{n(n+1)} + \frac{C}{(n+1)^2} \right] \\
&\qquad+ \Gamma (n+1)^{-\rho} \left[(1+n^{-1})^{\rho}\left(1-\frac{\gamma}{n+1}\right)^2 - 1\right] \\
&  \qquad+  \frac{\gamma^2V}{n(n+1)^2} + (A_1 \sqrt{V} n^{-2}+A_2\sqrt{V}n^{-r_1-1/2}) \left(1+\Gamma V^{-1} n^{-(\rho-1)}\right)^{1/2} + A_3 n^{-r_2} \\
& = \frac{V}{n+1} + \Gamma (n+1)^{-\rho} + \epsilon_n.
\end{align*}
We observe that for $n$ large enough, we have:
$$
\left(1+\Gamma V^{-1} n^{-(\rho-1)}\right)^{1/2} \leq 1 + \frac{\Gamma}{2 V} n^{-(\rho-1)} + O(n^{-2(\rho-1)}).
$$
Hence, we deduce that:
\begin{align*}
\epsilon_n & \leq \frac{V (1-2 \gamma) + A_1 \sqrt{V} + C }{n(n+1)} + \Gamma (n+1)^{-\rho} \left[(1+n^{-1})^{\rho}\left(1-\frac{\gamma}{n+1}\right)^2 - 1\right] \\
& +\frac{A_1\Gamma}{ 2\sqrt{V}}n^{-(\rho+1)}  + \frac{\gamma^2 V}{n(n+1)^2}+ A_3 n^{-r_2} + A_2\sqrt{V}n^{-r_1-1/2}+O(n^{-2\rho})+O(n^{-r_1-\rho+1/2}).
\end{align*}
We choose $\sqrt{V}$ as the positive root of $(1-2 \gamma) X^2+A_1 X+C$ and obtain that:
\begin{equation}\label{eq:defV}
\sqrt{V}= \frac{A_1}{2(2\gamma-1)} + \sqrt{\frac{A_1^2}{4(2\gamma-1)^2}+\frac{C}{2\gamma-1}},
\end{equation}
for which the leading term of $\epsilon_n$ vanishes. Then, we observe that
$$
 (1+n^{-1})^{\rho}\left(1-\frac{\gamma}{n+1}\right)^2 - 1 = 1+\frac{ \rho}{n}-\frac{2 \gamma}{n+1} - 1 + O(n^{-2}) = \frac{\rho-2 \gamma}{n} + O(n^{-2}).
$$
Hence, $\epsilon_n$ is upper bounded by:
\begin{align*}
\epsilon_n& \leq
\Gamma \left[ (\rho-2 \gamma)  +\frac{A_1}{2 \sqrt{V}} \right] n^{-(\rho+1)} +  \frac{\gamma^2 V}{n(n+1)^2} + A_3 n^{-r_2}+ A_2\sqrt{V}n^{-r_1-1/2}\\&\qquad+O(n^{-2\rho})+O(n^{-r_1-\rho+1/2})+O(n^{-(\rho+2)})
\end{align*}
%Using \eqref{eq:defV}, we have $\sqrt{V}\ge 2\frac{A_1}{2(2\gamma-1)} $ thus
%$$ \frac{A_1}{2\sqrt{V}} \leq \gamma - \frac{1}{2}.$$
%We then verify that:
%\begin{align*}
%\epsilon_n &\leq - \Gamma\left[\gamma-\rho+\frac{1}{2}\right]  n^{-(\rho+1)}+ \frac{\gamma^2 V}{n(n+1)^2} + A_2 n^{-r_2}+ A_2\sqrt{V}n^{-r_1-1/2}\\&\qquad+O(n^{-2\rho})+O(n^{-r_1-\rho+1/2})+O(n^{-(\rho+2)}).
%\end{align*}
Since $r_1>3/2$ and $r_2>2$, we then observe that as soon as $(\rho,\gamma)$ satisfies 
$$\rho = (r_2-1) \wedge \left(r_1-\frac{1}{2}\right) \wedge 2,\quad\text{and}\quad \gamma >\rho-\frac{1}{2}$$
then a large enough $\Gamma$ exists such that $\epsilon_n \leq 0$.
\end{proof}

\begin{lem}\label{lem:recurrence_finale_precis}
Let $\gamma>1/2$, $3>r>2$ and consider a positive sequence $u_n$ such that
$$
u_{n+1} \leq \left(1-\frac{1}{n+1}\right)^2 u_n + \frac{C}{(n+1)^2} + An^{-r}, 
$$
Then, there exists a large enough constant $\Gamma$ such that: 
\begin{equation*}%\label{eq:recursion}
\forall n\ge 1, \quad u_n \leq \frac{C}{n}+\frac{A}{3-r} n^{-(r-1)} + \Gamma n^{-r}.
\end{equation*}
\end{lem}
\begin{proof}
Let us proceed directly by recursion. Define $\pi_n=\prod_{k=1}^n (1-\frac{1}{k+1})^2=(n+1)^{-2}$, then 
$$u_{n+1}\le \pi_n u_0 +\sum_{k=1}^n \frac{C}{k+1}^2\frac{\pi_n}{\pi_k}+\sum_{k=1}^n A k^{-r}\frac{\pi_n}{\pi_k}.$$
Thus we deduce that: 
\begin{align*}
u_{n+1}&\le \frac{ u_0}{(n+1)^2} +\frac{Cn}{(n+1)^2}+\frac{A}{(n+1)^2}\sum_{k=1}^n  k^{-r}(k+1)^2.
\end{align*}
Now using a series-integral comparison 
$$\sum_{k=1}^n  k^{-r}(k+1)^2\le 4+ \frac{n^{3-r}}{3-r}+ 2\frac{n^{2-r}}{2-r}+\frac{n^{1-r}}{1-r},$$
which yields:
\begin{align*}
u_{n+1}&\le \frac{Cn}{(n+1)^2}+\frac{A}{(3-r)n^{1-r}} + \left[  \frac{ u_0}{(n+1)^2}+\frac{4A}{(n+1)^2}+2\frac{n^{-r}}{2-r}+\frac{n^{-r-1}}{1-r}\right].
\end{align*}
Let us  remark that the leading term within the brackets is of the order $n^{-r}$, which leads to the conclusion.
A similar reasoning also holds if the inequality is verified after an integer $n_0$.
\end{proof}
\section{Technical results for the quantile estimation }
\label{app:tech_tn}
\subsection{Technical results for  $(\tn)_{n\ge1}$ - Functions $(V_q)_{q \ge 1}$}
\begin{proof}[Proof of Lemma \ref{lem:prop_phi}]
A straigthforward computation from the Equation \eqref{def:Vq} yields:
\begin{equation}\label{eq:calculVq}
\begin{aligned}
&V_q'(\theta) = q V_{q-1}(\theta) \Phi'(\theta)+\Phi'(\theta) V_q(\theta) \\ 
&V_q''(\theta) =  V_{q-1}(\theta) \Phi''(\theta)[q + \Phi(\theta)]+\Phi'(\theta)^2 V_{q-1}(\theta)\left[\frac{q(q-1)}{\Phi(\theta)}+2q+\Phi(\theta)\right]
\end{aligned}\end{equation}
\underline{Proof of $i)$:} Using \eqref{eq:calculVq} we deduce that 
\begin{equation}
\label{eq:Lyapunov_calcul}
\Phi'(\theta)V_q'(\theta)=V_q(\theta)\left(q\frac{\Phi'(\theta)^2}{\Phi(\theta)}+\Phi'(\theta)^2\right)
\end{equation}
and we derive a lower bound for the bracket term $ q\frac{\Phi'(\theta)^2}{\Phi(\theta)}+\Phi'(\theta)^2$ on $\R$. This lower bound is deduced by a careful analysis in a neighborhood of $\ta$ and outside this small area.

\noindent\textbf{Local study}
 Using the definition \eqref{def:Phi}, we know that  $\Phi'(\ta)=0$.
A Taylor expansion yields:
$$
\Phi'(\theta) =  \Phi'(\ta)+ (\theta-\ta) \Phi''(\ta) + \frac{(\theta-\ta)^2}{2} \ \Phi'''(\xi_\theta),
$$
which implies since $\Phi''(\ta)=f(\ta)>0$ and $\Phi'''=f'$ that:
$$
\Phi'(\theta)^2 = (\theta-\ta)^2 f(\ta)^2 + (\theta-\ta)^3 f(\ta) f'(\xi_{\theta}) + \frac{(\theta-\ta)^4}{4} f'(\xi_\theta)^2.
$$
The Young inequality $|a b| \leq \frac{a^2}{2}+\frac{b^2}{2}$ leads to:
$$
\Phi'(\theta)^2 \ge  (\theta-\ta)^2 \frac{f(\ta)^2}{2} - \frac{(\theta-\ta)^4}{4} \|f'\|_{\infty}^2.
$$
We then define $\varepsilon = f(\ta) \|f'\|_{\infty}^{-1}$ and observe that
\begin{equation}\label{eq:minoration_locale}
\forall \theta \in [\ta-\varepsilon,\ta+\varepsilon] \qquad 
\Phi'(\theta)^2 \ge\frac{f(\ta)^2}{4}  (\theta-\ta)^2.
\end{equation}
In the meantime, we also verify that
$$
\forall \theta \in \R \qquad 
\Phi(\theta) = f(\ta) \frac{(\theta-\ta)^2}{2} + \|f'\|_{\infty} \frac{(\theta-\ta)^3}{6},
$$
so that:
\begin{equation}\label{eq:majoration_locale}
\forall \theta \in [\ta-\varepsilon,\ta+\varepsilon] \qquad \Phi(\theta) \leq \frac{ 2 f(\ta) }{3} (\theta-\ta)^2 .
\end{equation}
We gather \eqref{eq:minoration_locale}, \eqref{eq:majoration_locale} and use them in  Equation \eqref{eq:Lyapunov_calcul}
 to obtain that:

\begin{align*}
\forall \theta \in [\ta-\varepsilon,\ta+\varepsilon] \qquad \Phi'(\theta)  V_q'(\theta) & \ge q V_q(\theta) \frac{\Phi'(\theta)^2}{\Phi(\theta)}\\
& \ge \frac{3q}{8} f(\ta) V_q(\theta). 
\end{align*}
%
%\begin{align*}
%\Phi'(\theta)&=(\theta-\ta)f(\ta)+o((\theta-\ta)^2)\qquad \text{and}\qquad
%\Phi(\theta)=\frac{(\theta-\ta)^2}{2}f(\ta)+o((\theta-\ta)^3)
%\end{align*}
%since $\Phi'(\ta)=0$. We deduce that an $\varepsilon>0$ exists such that:
% \begin{equation}\label{eq:minoration}\forall \theta \in [\ta-\varepsilon,\ta+\varepsilon] \qquad\frac{\Phi'(\theta)^2}{\Phi(\theta)} > \frac{f(\ta)}{4}.\end{equation}

\noindent
\textbf{Far from $\ta$}
Remark that $\Phi'$ given by 
$\Phi'(\theta)=\int_{\ta}^{\theta} f(s)ds$ is increasing and negative on $]-\infty,\ta]$ and then positive on $[\ta,+\infty[$. In particular, we have that:
$$
\forall \theta \notin[\ta-\varepsilon,\ta+\varepsilon] \qquad {\Phi'}^2(\theta) \geq {\Phi'}^2\left(\ta+ \varepsilon \right) \wedge {\Phi'}^2\left(\ta- \varepsilon\right) .
$$
A straightforward computation associated with a first order Taylor expansion leads to:
\begin{align*}
\Phi'\left(\ta+ \varepsilon \right) &= \int_{\ta}^{\ta+ \varepsilon} f(s) \text{d}s \\
& = \int_{\ta}^{\ta+\varepsilon } f(\ta) + (s-\ta) f'(\zeta_s) \text{d}s \\
& \ge \int_{\ta}^{\ta+\varepsilon} f(\ta) - (s-\ta) f'(\zeta_s) \text{d}s \\
& \ge \varepsilon f(\ta) - \|f'\|_{\infty} \frac{\varepsilon^2}{2} = \frac{f(\ta)^2}{2 \|f'\|_{\infty}}.
\end{align*}
The same computations also apply for $\ta-\varepsilon$ and we conclude that:
$$
\forall \theta \notin[\ta-\varepsilon,\ta+\varepsilon] \qquad
\Phi'(\theta) ^2 \ge  \frac{f(\ta)^4}{4 \|f'\|_{\infty}^2}.
$$
We use again Equation \eqref{eq:Lyapunov_calcul} and the previous lower bound to deduce that:
$$
\forall \theta \notin [\ta-\varepsilon,\ta+\varepsilon] \qquad \Phi'(\theta)  V_q'(\theta) \ge \frac{f(\ta)^4}{4 \|f'\|_{\infty}^2} V_q(\theta).
$$

%As the sum of the two previous terms, we deduce that a real value $m>0$ exists such that
%$$
%\forall \theta \in \R \qquad q \frac{\Phi'(\theta)^2}{\Phi(\theta)} + \Phi'(\theta)^2 \geq m > 0.
%$$
\noindent\underline{Proof of $ii)$:}
Since $\Phi''=f$ and $V_q= \Phi V_{q-1}$, Equation \eqref{eq:calculVq} yields
\begin{align*}
|V_q''(\theta)| & \leq V_{q-1}(\theta) |f|_{\infty} \left[ q +\Phi(\theta)\right]+\Phi'(\theta)^2 V_{q-1}(\theta)\left[\frac{q(q-1)}{\Phi(\theta)}+2q+\Phi(\theta)\right] \\
& = V_q(\theta)\left(|f|_{\infty}+2q\frac{\Phi'(\theta)^2}{\Phi(\theta)}+\Phi'(\theta)^2\right) \\ 
&\quad + V_{q-1}(\theta)\left(q |f|_{\infty} + q(q-1)\frac{\Phi'(\theta)^2}{\Phi(\theta)} \right).
 \end{align*}	
We have shown in $i)$ that $\Phi'^2 \Phi^{-1}$ is bounded near $\ta$ whereas the unique zero of $\Phi$ is $\ta$ and $\Phi'$ satisfies $|\Phi'|_{\infty} \leq 1$. Consequently, $\Phi'^2 \Phi^{-1}$ is a bounded function. More precisely, we shall observe that:
$$
\Phi'(\theta)^2 \leq \frac{9}{4} (\theta-\ta)^2 f(\ta)^2 \mathbf{1}_{|\theta-\ta|\leq \varepsilon} + \mathbf{1}_{|\theta-\ta|\ge \varepsilon}
$$
Using that 
$$
\forall \theta \notin [\ta-\varepsilon,\ta+\varepsilon] \quad \Phi(\theta) \ge \Phi(\ta+\varepsilon) \wedge \Phi(\ta-\varepsilon),
$$
and considering $\ta+\varepsilon$, we have:
\begin{align*}
\Phi(\ta+\varepsilon)& = \int_{\ta}^{\ta+\varepsilon} \int_{\ta}^u f(s)  \text{d} s
\text{d}u\\
& \ge  \int_{\ta}^{\ta+\varepsilon} \int_{\ta}^u f(\ta) - (s-\ta) \|f'\|_{\infty}  \text{d} s
\text{d}u\\
& \ge  \frac{f(\ta)}{2} \int_{\ta}^{\ta+\varepsilon} (u-\ta) \text{d}u \ge \frac{f(\ta)}{4} \varepsilon^2 = \frac{f^{3}(\ta)}{4  \|f'\|_{\infty}^2}
\end{align*}
Using the same argument, we observe that
$$
\forall \theta \in [\ta-\varepsilon,\ta+\varepsilon] \qquad 
\Phi(\theta) \ge \frac{f(\ta)}{4} (\theta-\ta)^2
$$
Finally
$$
\Phi'(\theta)^{2} \Phi(\theta)^{-1} \leq \underbrace{9 f(\ta) \vee 4 \frac{\|f'\|_{\infty}^2}{f^3(\ta)}}_{:=C_\alpha}
$$
We then deduce that $c_q>0$ exists such that
$$
V_q''(\theta) \leq c_q [V_q(\theta)+V_{q-1}(\theta)] 
$$
and $c_q=1+q \|f\|_{\infty} + q[(q-1) \vee 2] C_{\alpha}$.

\noindent\underline{Proof of $iii)$}
It remains to derive a lower bound of $\Phi(\theta)$ that involves powers of $(\theta-\ta)$.
We consider the same $\varepsilon$ and use that $\forall \theta\in[\ta-\varepsilon,\ta+\varepsilon]$, $\Phi(\theta)\ge \frac{f(\ta)}{2}\frac{(\theta-\ta)^2}{2}$ and $\Phi'(\theta)\ge \frac{f(\ta)}{2}(\theta-\ta)$.
Then we have for $\theta \geq \ta$:
\begin{eqnarray*}
\Phi(\theta)&=&\int_{\ta}^\theta \int_{\ta}^u f(s) \text{d} s \text{d} u \\
& = & \int_{\ta}^{\theta \wedge (\ta+\varepsilon)} \int_{\ta}^u f(s) \text{d} s \text{d} u+ \int_{\theta \wedge (\ta+\varepsilon)}^{\theta } \int_{\ta}^u f(s) \text{d} s \text{d} u
\end{eqnarray*}
Differentiating whether $\theta>\ta+\varepsilon$ or not, and using the previous bound we deduce
\begin{eqnarray*}
\Phi(\theta)
& \geq & \un_{\theta \leq \ta+\varepsilon} \frac{f(\ta)}{2} \frac{(\theta-\ta)^2}{2}
 + \un_{\theta > \ta+\varepsilon} \frac{f(\ta)}{2} \frac{\varepsilon^2}{2}
 + \un_{\theta > \ta+\varepsilon} \int_{\ta+\varepsilon}^{\theta} \int_{\ta}^{\ta+\varepsilon} f(s)d(s)\text{d}s\text{d}u\\
 & \geq  & \frac{f(\ta)}{2} \left[ \un_{\theta \leq  \ta+\varepsilon}  \frac{(\theta-\ta)^2}{2}  + 
\un_{\theta > \ta+\varepsilon} \left[\frac{\varepsilon^2}{2} + \varepsilon (\theta-\ta-\varepsilon)\right]
 \right] \\
 & \geq &\frac{f(\ta)}{2} \left[ \un_{\theta \leq  \ta+\varepsilon}  \frac{(\theta-\ta)^2}{2} + \un_{\theta > \ta+\varepsilon} \left[
\varepsilon (\theta-\ta) - \frac{\varepsilon^2}{2} 
  \right] \right]\\
  & \geq& \frac{f(\ta)}{2} \left[ \un_{\theta \leq  \ta+\varepsilon}  \frac{(\theta-\ta)^2}{2} + \un_{\theta > \ta+\varepsilon} \left[
\frac{\varepsilon}{2} (\theta-\ta) )
  \right] \right],
\end{eqnarray*}
where the last line comes from the fact that when $\theta \ge \ta+\varepsilon$:
$ \frac{\varepsilon}{2} (\theta-\ta)- \frac{\varepsilon^2}{2} \ge 0$
We finally deduce that: 
\begin{eqnarray*}
\Phi(\theta)
& \geq  &\frac{f(\ta)}{2} \left[ \un_{\theta \leq  \ta+\varepsilon}  \frac{(\theta-\ta)^2}{2} + \un_{\theta > \ta+\varepsilon}   
\frac{\varepsilon (\theta-\ta)}{2}.
 \right]
\end{eqnarray*} 
We repeat the same arguments when $\theta \leq \ta$ and deduce that :
\begin{eqnarray*}
(\theta-\ta)^{2q} &= &(\theta-\ta)^{2q}\un_{|\theta -  \ta| \leq \varepsilon} +
(\theta-\ta)^{2q}\un_{|\theta -  \ta| > \varepsilon}\\
& \leq &   \frac{4^q}{f(\ta)^q} \Phi(\theta)^q
\un_{|\theta -  \ta| \leq \varepsilon} + \frac{4^{2q} \|f'\|_{\infty}^{2q}}{f(\ta)^{4 q}} \Phi(\theta)^{2q}
\un_{|\theta -  \ta| > \varepsilon} \\
& \leq & C_q \left(  V_{q}(\theta)+V_{2q}(\theta)\right).
\end{eqnarray*}
\end{proof}

\begin{lem}
\label{lem:recursion_Vq}
Let us consider the function $V_q$ introduced in \eqref{def:Vq} and recall that there exists $n_0>0$ such that for any $q\ge1$, a $c_q>0$ exists such that:
\begin{align*}
\forall n \ge n_0 \qquad 
\E [ V_q(\tnp)]
& \leq \E \left( V_q(\tn)\right) \left(1-a_{n} m  + c_q a_{n}^2 \right) + c_q a_{n}^2 \E \left( V_{q-1}(\tn)\right)  + c_q a_{n}^{q+1}. 
\end{align*}
Then for any $q\ge1$ there exists $(A_q,B_q,C_q)\ge0$ such that 
$$\E [ V_q(\tn)] \le \exp(A_q-B_q n^{1-a}) + C_q a_n^q,\quad \forall n\ge1.$$
In particular, 
$$
A_1 = \frac{4 a_1^2 c_1 }{2a-1}+\log(2),  \quad B_1= \frac{a_1 m}{2} \quad \text{and} \quad C_1= \frac{8  c_1}{m}
$$
and
$$
A_2 = \frac{c_1^2 a_1}{1-a} +\log(2),  \quad B_2=\frac{a_1 m}{2} \quad \text{and} \quad C_2= \frac{4}{m}\left(c_2+c_2 C_1 + \frac{e^{A_1}}{ a_1}\right).
$$
\end{lem}
\begin{proof}[Proof of Lemma \ref{lem:recursion_Vq}]
\underline{Initialization for $ q=1$:} We prove the result by recursion on $n$. Let us first remark that $V_0(\theta)\le e+V_1(\theta)$, then \eqref{eq:recurrence} becomes
\begin{align}
\E [ V_1(\tnp)]  &\leq \E \left( V_1(\tn)\right) \left(1-a_{n} m  + c_1 a_{n}^2 \right) + c_1 a_{n}^2 \E \left( e+V_1(\tn)\right)  + c_1 a_{n}^{2}\nonumber\\
&\le \E \left( V_1(\tn)\right) \left(1-a_{n} m  + 2c_1 a_{n}^2 \right) + c_1(1+e) a_{n}^{2}.
\label{eq:super_important}
\end{align}
We then apply the same argument as the one used in Theorem 1 of  \cite{Bach_Moulines_2011} and deduce that
$$
\forall n \ge 0 \qquad \E V_1(\tn) \leq 2 \exp\left(\frac{4 a_1^2 c_1}{2a-1} - \frac{a_1 m}{2} n^{1-a}\right) (V_1(0)+2)+ \frac{8 a_1 c_1}{m} n^{-a}.
$$

\noindent which concludes the recursion and prove the property for $q=1$ by defining
$$
A_1 = \frac{4 a_1^2 c_1 }{2a-1}+\log(2),  \quad B_1= \frac{a_1 m}{2} \quad \text{and} \quad C_1= \frac{8  c_1}{m}.
$$

\noindent
\underline{Recursion step:} We assume that the property holds until the integer $q-1$ and we prove similarly by induction on $n$ that the property holds for the integer $q$. 
We then obtain that 
\begin{align*}
\E V_q(\tnp)&\leq (1- m a_{n} + c_1 a_n^2) \E V_q(\tn) + (c_q+c_q C_{q-1}) a_n^{q+1} + a_n^{2}  e^{A_{q-1}-B_{q-1} n^{1-a}}
\end{align*}
We then use the (rough) upper bound:
$$
\E V_q(\tnp) \leq (1- m a_{n} + c_1 a_n^2) \E V_q(\tn) + \left(c_q+c_q C_{q-1}+\frac{e^{A_{q-1}}}{a_1^{q-1}}\right) a_n^{q+1}
$$
We then repeat the arguments of \cite{Bach_Moulines_2011} and deduce that 
$$
\forall n \ge 0 \qquad \E V_q(\tn) \leq 2 \exp\left(A_{q} - B_q  n^{1-a}\right) (V_q(0)+2)+ C_q a_n^q,
$$
where $(A_q,B_q,C_q)$ are given by:
$$
A_q= \frac{c_1^2 a_1}{1-a} + \log(2),\quad B_q = \frac{a_1 m}{2}, \quad \text{and} \quad C_q=\frac{4}{m} \left(c_q+c_q C_{q-1} + e^{A_{q-1}} a_1^{-(q-1)}\right).
$$

\end{proof}

\subsection{Technical results for the averaged quantile estimation $(\btn)_{n \geq 1}$}
\label{app:lem_quantile}
This paragraph gathers technical results useful for the study of $(\btn)_{n \geq 1}$.

\begin{proof}[Proof of Lemma \ref{lem:t1}]
A straigthforward computation yields:
\begin{align*}
a_{n}^2  \left( \epsilon_{n+1}-\frac{1}{n+1}\right)^2 & = a_{n}^2 \left( \frac{1-a_{n} f(\ta)}{1-(n+1)a_{n}f(\ta)}-\frac{1}{n+1} \right)^2\\\nonumber
& = \frac{a_{n}^2 }{(n+1)^2 a_{n}^2 f(\ta)^2} \left[\frac{1-a_{n} f(\ta)}{\frac{1}{(n+1) a_{n} f(\ta)} - 1} - a_{n} f(\ta) \right]^2 \\\nonumber
& = \frac{a_{n}^2 }{(n+1)^2 a_{n}^2 f(\ta)^2} \left[\frac{\frac{n}{n+1}}{1-\frac{1}{(n+1) a_{n} f(\ta)} }  \right]^2 \\\nonumber
 & = \frac{1}{(n+1)^2 f(\ta)^2} \left[ 1+\frac{1}{(n+1)a_nf(\ta)}\frac{1-a_nf(\ta)}{1-\frac{1}{(n+1)a_nf(\ta)}} \right]^2.
\end{align*}
We then conclude that:
\begin{equation}
\label{eq:maj1}
a_{n}^2 \left( \epsilon_{n+1}-\frac{1}{n+1}\right)^2 \leq \frac{1}{(n+1)^2 f(\ta)^2}  + \frac{2}{f(\ta)^3a_1} n^{-(3-a)} + o(n^{-(3-a)}) .
\end{equation}
We recall that from Proposition \ref{prop:lin_btn}:
$$
\left|\E \left[\{\Delta M_{n+1}\}^2\vert \Fn\right]-\alpha(1-\alpha)\right| \leq \|f\|_{\infty} |\ta-\tn|.
$$
Then, the Cauchy-Schwarz inequality yields $\E|\ta-\tn| \leq \sqrt{\E|\ta-\tn|^2}$ so that:
\begin{align*}
\E \{\Delta M_{n+1}\}^2
a_{n}^2 \left( \epsilon_{n+1}-\frac{1}{n+1}\right)^2 
&\leq \left(\frac{1}{(n+1)^2 f(\ta)^2}  + \frac{2}{f(\ta)^3a_1} n^{-(3-a)} + o(n^{-(3-a)})\right)\\
&\quad\times\left(\alpha (1-\alpha ) +  C n^{-a/2}\right),
\end{align*}
from which we deduce the lemma.\end{proof}
The second lemma handles the rests terms given in \eqref{eq:restes}.
\begin{proof}[Proof of Lemma \ref{lem:t2}]
Using the Jensen inequality as well as Proposition \ref{prop:lin_btn} $iii)$, we deduce that:
$$\E(\mathcal{R}_n^2)\le 2\omega_n^2 \E((\tn-\ta)^2)+\frac{a_n^2L^2}{2}(\frac{1}{n+1}-\epsilon_{n+1})^2 \E (\tn-\ta)^4).$$
For the first term, a direct computation shows that:
\begin{equation}
\label{eq:omega}
\omega_n= \frac{-(a+1)}{n^2a_n} + o(n^{-2+a}),
\end{equation}
which entails:
$$
\omega_{n}^2 \lesssim (\epsilon_{n+1}-\epsilon_{n})^2  \lesssim (n+1)^{-4+2a}.
$$
Now, we apply Theorem \ref{theo:momentp} and conclude that
$$
\omega_{n+1}^2 \E (\tn-\ta)^2 \lesssim (n)^{-4+a}.
$$
For the second term, we combine the bound on $a_n^2(1/(N+1)-\epsilon_{n+1})$ obtained in \eqref{eq:maj1} with Theorem \ref{theo:momentp} and deduce that: 
$$\frac{a_n^2L^2}{2}(\frac{1}{n+1}-\epsilon_n)^2 \E (\tn-\ta)^4)\lesssim n^{-2-2a}.$$
Combining the two bounds, we deduce that:
$$\E(\mathcal{R}_n^2)\lesssim n^{-((4-a)\wedge(2+2a))}.$$
\end{proof}

\begin{proof}[Proof of Theorem \ref{theo:rate_averaging}-$ii)$]

We use the key relationship introduced in \cite{pelletier2000} derived from the linearization:
$$
\tnp-\tn = - a_n f(\ta) (\tn-\ta) + a_n R_{n} + a_n \Delta M_{n+1}.
$$
Dividing by $a_n$ and summing, we then obtain that:
$$
\sum_{k=1}^n \left(\frac{\theta_{k+1}}{a_k} - \frac{\theta_k}{a_k} \right)=- f(\ta) \sum_{k=1}^n (\theta_k-\ta) + \sum_{k=1}^n R_{k} + \sum_{k=1}^n \Delta M_{k+1}.
$$
Finally, dividing by $n$ and using an Abel transform, we verify that:
\begin{align}
f(\ta) [\btn-\ta]  &= \frac{1}{n} \left( \frac{\theta_1-\ta}{ a_1} - 
\frac{\tnp-\ta}{ a_{n+1}} + \sum_{k=2}^{n+1} \left(\frac{1}{a_k} - \frac{1}{a_{k-1}}\right) (\theta_k-\ta) + \sum_{k=1}^n R_{k}\right)  \nonumber\\ 
& \qquad+ \frac{1}{n} \sum_{k=1}^n \Delta M_{k+1}. \label{eq:decomposition_pelletier}
\end{align}
The Jensen inequality 
 applied to \eqref{eq:decomposition_pelletier} yields:
\begin{align*}
|\btn-\ta|^{2p} &\leq n^{-2p} 5^{2p-1} \left(\left|\frac{\theta_1-\ta}{a_1}\right|^{2p}+ \left|\frac{\theta_{n+1}-\ta}{a_{n+1}}\right|^{2p} + \left|\sum_{k=2}^{n+1} \left(\frac{1}{a_k} - \frac{1}{a_{k-1}}\right) (\theta_k-\ta)\right|^{2p} \right. \\
& \left.
+\left|\sum_{k=1}^n R_{k+1}\right|^{2p} + \left|\sum_{k=1}^n \Delta M_{k+1}\right|^{2p}
 \right)
\end{align*}
We then compute the expectation of the right hand side terms and remark that:
\begin{equation}\label{eq:term1}
\E \left|\frac{\theta_1-\ta}{a_1}\right|^{2p}  = O(1),
\end{equation}
while Theorem \ref{theo:momentp} yields:
\begin{equation}\label{eq:term2}
\E \left|\frac{\theta_{n+1}-\ta}{a_n}\right|^{2p} = O(a_n^{-p}) = o(n^{p}),
\end{equation}
since $a_n = a_1 n^{-a}$ with $a \in (0,1)$.\\
The generalized Hold\"er inequality yields (see \textit{e.g.} Lemma 4.3 of \cite{godichon2015})
$$
\E \left|\sum_{k=2}^n \left(\frac{1}{a_k} - \frac{1}{a_{k-1}}\right) (\theta_k-\ta)\right|^{2p}  \leq 
\left|\sum_{k=2}^n \left(\frac{1}{a_k} - \frac{1}{a_{k-1}}\right) \left(\E|\theta_k-\ta|^{2p}\right)^{1/2p}  \right|^{2p}.
$$
When $a_k=a_1 k^{-a}$, we can verify that $\frac{1}{a_k} - \frac{1}{a_{k-1}} \lesssim k^{-1+a}$. Using again Theorem
\ref{theo:momentp}, we have
$$
\left(\E|\theta_k-\ta|^{2p}\right)^{1/2p} \lesssim \left(a_n^{p}\right)^{1/2p} = \sqrt{a_n} \lesssim n^{-a/2}.
$$
Therefore a comparison between series and integral implies that:
\begin{equation}\label{eq:term3}
 \E \left|\sum_{k=2}^n \left(\frac{1}{a_k} - \frac{1}{a_{k-1}}\right) (\theta_k-\ta)\right|^{2p}  \lesssim  \left(\sum_{k=2}^n k^{-1+a/2}\right)^{2p} \lesssim n^{a p}.
\end{equation}

Concerning the rest terms, we use that $|R_{k+1}| \lesssim |\theta_k-\ta|^2$ (see Proposition \ref{prop:lin_btn} $iii)$), Theorem \ref{theo:momentp} and we deduce from the generalized Holder inequality that:
\begin{align}
\E \left|\sum_{k=1}^n R_{k+1}\right|^{2p} &\leq
\left( \sum_{k=1}^n \left(\E |R_{k+1}|^{2p} \right)^{1/2p}\right)^{2p} \nonumber \\
%& \lesssim \left( \sum_{k=1}^n \left(\E |\theta_k-\ta|^{4p} \right)^{1/2p}\right)^{2p}\nonumber\\
&\lesssim \left( \sum_{k=1}^n \left(a_n^{2p} \right)^{1/2p}\right)^{2p}  \lesssim \left( \sum_{k=1}^n a_n\right)^{2p} \lesssim n^{2(1-a)p} \label{eq:term4}.
\end{align}
We remark that $(\Delta M_{k+1})_{k \geq 1}$ is a \textit{bounded} sequence of martingale increments. We shall apply the recursive argument stated in Lemma A.2 of \cite{godichon2015}: a constant $K_p$ exists such that:
\begin{equation}\label{eq:Lpmartingale}
\forall n \geq 1 \qquad \E \left[ \left( \sum_{k=1}^n \Delta M_{k+1} \right)^{2p}\right] \leq K_p n^p.
\end{equation}
We then gather Equations \eqref{eq:term1}-\eqref{eq:Lpmartingale} and obtain that:
$$
\E |\btn-\ta|^{2p} \lesssim n^{-2p} [ O(1)+O(n^{p})+O(n^{ap})+O(n^{2(1-a)p}] = O(n^{-p}),
$$
for $a\in[1/2,1)$, which concludes the proof.
\end{proof}

\begin{proof}[Proof of Lemma \ref{lem:cov}]
We use the definition of $\mathcal{R}_n$ and $\tZnd$ and write that:
\begin{equation}
\label{eq:dev_reste}
E(\mathcal{R}_n\tZnd)=\omega_n\E((\tn-\ta)\tZnd) + a_n\left(\frac{1}{n+1}-\epsilon_{n+1}\right) \E(R_n \tZnd).\end{equation}
\underline{Step 1: Recursive study of the covariance term}\\
Using \eqref{eq:rec-tZn}, we obtain a recursion on this covariance:  
\begin{align*}
\E((\tnp-\ta)\tZnpd)&=\left(1-\frac{1}{n+1}\right) (1-a_nf(\ta)) \E((\tn-\ta)\tZnd) + a_n^2\left(\frac{1}{n+1}-\epsilon_{n+1}\right)\E(\Delta M_{n+1}^2) \\
&+a_n\E(R_n\tZnd)\left(1-\frac{1}{n+1}\right) + \Omega_{n+1},
\end{align*}
where: 
\begin{align*}
\Omega_{n+1}&= \left[a_n\omega_n+\frac{a_n n}{n+1}(1-a_nf(\ta))\right]\E(R_n(\tn-\ta))+ (1-a_nf(\ta)\omega_n \E((\tn-\ta)^2)\\ &+a_n^2\left(\frac{1}{n+1}-\epsilon_{n+1}\right)\E(R_n^2).\end{align*}
Using \eqref{eq:reste_1} and Theorem \ref{theo:momentp}, we deduce that 
$$\Omega_{n+1}=O(n^{-2}).$$
Using first the Cauchy Schwarz inequality as well as \eqref{eq:maj_un1} and then Theorem \ref{theo:momentp}, we deduce that: 
$$a_n\E(R_n\tZnd)\le a_n\E(\tn-\ta)^4)^{1/2} \E(\{\tZnd\}^2)^{1/2}\le n^{-2a-1/2}.$$
Combining with Lemma \ref{lem:t1}
we finally deduce that:
\begin{align*}
\E((\tnp-\ta)\tZnpd)&\le \frac{n}{n+1} (1-a_nf(\ta))\E((\tn-\ta)\tZnd) \\&+ \frac{\alpha(1-\alpha)a_n}{f(\ta)^2(n+1))}+O(n^{-2a-1/2}\vee n^{-1-3a/2}\vee n^{-2}).
\end{align*}
Studying the recursion satisfied by $n\E((\tn-\ta)\tZnd) $ and using similar arguments as in Lemma \ref{lem:recurrence_finale_3}, we deduce that there exists a large enough constant $C$ such that: 
\begin{equation*}
%\label{eq:maj_cov}
\E((\tn-\ta)\tZnd) \le \frac{\alpha(1-\alpha)}{f(\ta)^3n} + Cn^{-(1\wedge \frac{3a}{2}\wedge (2a-1/2))}.
\end{equation*}
\noindent\underline{Step 2 : Upper-bound of the second term in \eqref{eq:dev_reste}}
Using \ref{eq:reste_1} as well as the Cauchy-Schwarz inequality we obtain that:
\begin{align*}
 \E(R_n \tZnd)&\le  \frac{L}{2}\E((\tn-\ta)^4)^{1/2}\E(\{\tZnd\}^2)^{1/2}.
\end{align*}
Now combining these bounds with Theorem \ref{theo:momentp}, \eqref{eq:maj1} and \eqref{eq:maj_un1} we deduce that:
$$a_n\left(\frac{1}{n+1}-\epsilon_{n+1}\right) \E(R_n \tZnd)\le \frac{L\alpha(1-\alpha)}{2f(\ta)^3(n+1)^{3/2}}K_2 a_n + C n^{-\left(4\wedge (\frac{5}{2}+\frac{a}{2})\wedge \left(\frac{3}{2}+2a\right)\right)}.$$
\underline{Step 3: Conclusion}
\\Combining the two previous bounds with the majoration of $\omega_n$ given by \eqref{eq:omega}, we deduce that: 
$$\E(\mathcal{R}_n\tZnd)\le \frac{-(a+1)}{n^2a_n}  \frac{\alpha(1-\alpha)}{f(\ta)^3n} + \frac{L\alpha(1-\alpha)}{2f(\ta)^3}K_2 a_n n^{-3/2}+ o(n^{-3+a}\vee n^{-3/2-a}).$$
\end{proof}

\section{Technical results for the super-quantile algorithm $(\bvtn)_{n \geq 1}$}
\label{app:calcul_sq}
This paragraph gives the missing details of the proof of Proposition \ref{prop:W_n}.

\begin{lem}\label{lem:tec_etan} Assume that the step-size sequences satisfy $$
a_n =a_1 n^{-a} \qquad \text{and} \qquad b_n = b_1 (n+1)^{-b}\qquad \text{with } 1/2<a<b\le 1.
$$ then a constant $C_\alpha$ exists such that:
$$
|\eta_n| \leq C_{\alpha} n^{-2-b+2a}.$$
\end{lem}
%\no{Ici il y a une différence quand $a>b$ car dans ce cas $\delta_n$ converge vers une constante. Du coup si je ne me plante pas trop, 
%$$|\eta_n| \le n^{-(2-a)\wedge(1+a-b)}$$}
\begin{proof}
We recall from Section \ref{sec:quantile} that a large enough constant $C_\alpha$ exists such that:
$$
|\epsilon_n| \leq \frac{C_\alpha}{n a_n f(\ta)}\qquad \text{and}\qquad |\epsilon_n-\epsilon_{n+1}| \leq  C_\alpha n^{-2+a}.
$$
In the meantime, we can check that a large enough $C_\alpha$ exists such that:
$$
|\delta_n| \leq C_\alpha \frac{\ta }{1-\alpha} n^{-b+a} \qquad \text{and}\qquad |\delta_n-\delta_{n+1}| \leq C_\alpha n^{-1-b+a}.
$$
The triangle inequality yields:
\begin{align*}
|\eta_n| & \leq |\epsilon_{n+1}-\epsilon_{n}|  |\delta_{n+1} | + |\epsilon_n||\delta_n-\delta_{n+1}|  \leq C_\alpha    n^{-2-b+2a}.
\end{align*} 
%The sequence $(v_n)_{n \geq 1}$ is defined in \eqref{eq:v_n} and  using Lemma \ref{lem:t1}, we obtain that:
%\begin{align*}
%| v_{n+1}a_n| & \le \kanp a_n\left|\epsilon_{n+1}-\frac{1}{n+1}\right| +\delta_{n+1}|\epsilon_{n+1}| a_n\\
%& \leq \frac{\ta}{n(1-\alpha)}(1+O(n^{-1+a})) +Cn^{-2+2a}.
%\end{align*}
\end{proof}

We now prove the two lemma contained in the proof of Proposition \ref{prop:W_n}.
\begin{proof}[Proof of Lemma \ref{lem:tec_termes_carres}]
Recall that $U_n$ defined in \eqref{def:U_n} is the sum of four terms so that:
%. The Jensen inequality  $(x_1+x_2+x_3+x_4)^2 \le 4 (x_1^2 + x_2^2+x_3^2 + x_4^2)$ shows that 
$$\E(U_n^2)\le 4 \E \left[ b_n^2 |\tconn{2}|^2 + \eta_n ^2 
 |\tconn{1}|^2  +\epsilon_{n+1}^2 \delta_{n+1}^2 a_n^2  R_{n}^2 + b_n^2 \tilde{R}_{n}^2\right] .
$$
We consider each term of the previous sum separately. \\We remind that $\tconn{1}  = \theta_n-\ta$. Then,  Lemma \ref{lem:tec_etan} and Theorem \ref{theo:momentp} yield:
$$\eta_n^2  \E(\{\tconn{1}\}^2)\lesssim n^{-4-2b+2a}n^{-a}=n^{-4-2b+a}.$$
Concerning the second term, remark that $\tconn{2}=\tZnd$, so that \eqref{eq:maj_un1} entails:
$$
b_n^2  \E(\{\tconn{2}\}^2) \lesssim n^{-2b-1}.
$$
For the rest term $R_{n}$, we use that:
$R_{n} \leq \frac{L}{2} |\tn-\ta|^2$ and Theorem \ref{theo:momentp}, so that:
$$
\epsilon_{n+1}^2 \delta_{n+1}^2 a_n^2 \E[ R_{n}^2] \lesssim n^{-2+2a} \times n^{-2+2a} \times n^{-2a} \times n^{-a} \lesssim n^{-4+a}.
$$
Finally, for the last term, we use that $|\tilde{R}_{n}| \leq \frac{\|Q''\|_{\infty}}{2(1-\alpha)} |\btn-\ta|^2$. Then, Theorem \ref{theo:momentp} leads to:
$$
b_n^2 \E \tilde{R}_{n}^2 \lesssim b_n^2 \E[|\btn-\ta|^4] \lesssim n^{-2b-2}.
$$
Gathering all the previous bounds and using $1/2 < a < b \leq 1$, we get:
$$\E(U_n^2)\lesssim b_n^2(n^{-1}+n^{-2})+n^{-4+a}+n^{-4-2b+3a} \lesssim n^{-(1+2b)}.$$
\end{proof}

\begin{proof}[Proof of Lemma \ref{lem:tec_termes_martingales}]
We study the variance terms brought by $\Delta\mathcal{V}_n$:
$$
\E(\Delta\mathcal{V}_{n+1}^2)  = \epsilon_{n+1}^2 \delta_{n+1}^2 a_n^2 \E[ \{\Delta M_{n+1}\}^2] + \frac{b_n^2}{(1-\alpha)^2} \E[\{\Delta N_{n+1}\}^2 ] - 2 \frac{\epsilon_{n+1} \delta_{n+1} a_n b_n }{1-\alpha} \E[\Delta M_{n+1} \Delta N_{n+1}],
$$
and study each term separately.\\
\noindent\underline{First term.}
From Lemma \ref{lem:t1}, $\E(\{\Delta M_{n+1}\}^2)\le \alpha(1-\alpha)+Cn^{-a/2}$. Combining this with the upper bound of $\epsilon_n\delta_n=n^{-2+2a}$ in Lemma \ref{lem:tec_etan}, we observe that:
\begin{align*}
\E [(a_n\epsilon_{n+1}\delta_{n+1})^2 \{\Delta M_{n+1}\}^2] %&\le \left(\frac{\ka}{f(\ta)}+O(n^{-2+a})\right)^2(\alpha(1-\alpha)+Cn^{-a/2})\\
&\le C n^{-4+2a}.
\end{align*}
\noindent\underline{Second term.}
We use \eqref{eq:martingale_variance} and Theorem \ref{theo:momentp} to deduce that:
$$
\E \left[\left( \frac{b_n}{1-\alpha} \Delta N_{n+1}\right)^2 \right] = \frac{b_n^2}{(1-\alpha)^2} \left[V_{\alpha} + O\left(\sqrt{\MM_{n,2}}\right)\right] =  \frac{b_n^2V_\alpha}{(1-\alpha)^2}+ O\left(b_n^2a_n^{1/2}\right).
$$
\noindent\underline{Third term.} We verify with simple algebra that:
\begin{align*}
\E[ \Delta M_{n+1} \Delta N_{n+1} \vert \Fn] & = \E \left[ (X_{n+1} \un_{X_{n+1} \geq \btn} - Q(\btn)) ( F(\tn)-\un_{X_{n+1} \leq \tn}) \vert \Fn\right] \\
& =  [F(\tn) Q(\btn)] - \E \left[ X_{n+1} \un_{X_{n+1} \geq \btn} \un_{X_{n+1} \leq \tn}\vert \Fn \right] \\
& =  \alpha (1-\alpha) \vta + \E [ F(\tn) Q(\btn)-F(\ta)Q(\ta)]
 - \E \left[ X_{n+1} \un_{X_{n+1} \geq \btn} \un_{X_{n+1} \leq \tn} \vert \Fn\right].
\end{align*}
We apply the Cauchy-Schwarz inequality and obtain that:
$$
\E \left[ X_{n+1} \un_{X_{n+1} \geq \btn} \un_{X_{n+1} \leq \tn}\vert \Fn \right] \leq \sqrt{\E[X^2]} \sqrt{|F(\tn)-F(\tn)|}.
$$
We deduce that a large enough constant exists such that:
$$
\left|\E[ \Delta M_{n+1} \Delta N_{n+1} \vert \Fn] - \alpha(1-\alpha)\vta\right| \leq C \left[\sqrt{|\btn-\ta|}+\sqrt{|\tn-\ta|}+|\btn-\ta|+|\tn-\ta|\right].
$$
In particular, we verify that:
$$
\E[ \Delta M_{n+1} \Delta N_{n+1}] = \alpha (1-\alpha) \vta + O(a_n^{1/4}).
$$

We obtain with simple algebra using Lemma \ref{lem:tec_etan} that:
$$\E \left[2\frac{b_n}{1-\alpha} \epsilon_{n+1}\delta_{n+1} a_n\{\Delta M_{n+1}\} \{\Delta N_{n+1}\} \right]= O(n^{-3+a}).$$
To conclude we gather the three upperbounds and deduce that:
$$\E(\Delta\mathcal{V}_{n+1}^2)= b_n^2\frac{V_\alpha}{(1-\alpha)^2}  + O\left(b_n^2a_n^{1/2}\vee n^{-(3-a)}\right).$$
\end{proof}

\section{Technical results for the central limit theorem}\label{app:tcl}

\begin{proof}[Proof of Proposition \ref{prop:tightness}]
\underline{Proof of $i)$:}
The proof relies on Theorem 7.3 of \cite{billingsley} and is divided into three steps.

\noindent
\underline{Step 1: Tightness of $(\Tztn_0)_{n \geq 1}$.}
We shall remark that 
$$
\sup_{n \geq 1} \E[ \|\Tzn\|^2] \leq \sup_{n \geq 1} \E \tTn ^2 + \sup_{n \geq 1} \E \Vtn^2.
$$
We observe that $\tTn=\sqrt{n}\tZn^{(2)}$ and \eqref{eq:maj_un1} implies: $\sup_{n\ge1}E(\tTn^2)=\sup_{n\ge1}n\E\{\tZn^{(2)}\}^2<+\infty$.

In the meantime, we deal with the initialization of $\Vtn$ 
by applying Theorem \ref{theo:rate_vtn}. We then deduce that:
$
\sup_{n \geq 1} \E \Vtn^2 < + \infty,
$
so that 
$
\sup_{n \geq 1} \E \| \Tzn\|^2 < + \infty.
$
We then conclude that $(\Tztn_0)_{n \geq 1}=(\Tzn)_{n\ge1}$ is a tight sequence.
\medskip

\noindent To prove the tightness of  $(\Tztn)_{n \geq 1}$, 
we follow Theorem 7.3 of \cite{billingsley} and 
we are led to study $\mathbb{P}\left(\sup_{s\in[t,t+\delta]}|\Tztn_s-\Tztn_t| \geq \epsilon \right)$. This difference will be decomposed as:
$$
\Tztn_s-\Tztn_t = B_s^{(n)}-B_t^{(n)} + M_s^{(n)}-M_t^{(n)}.
$$

\noindent
\underline{Step 2: Study of $B_s^{(n)}-B_t^{(n)}$ - Tightness of $(B^{(n)})_{n \geq 1}$.}
We consider $\epsilon>0$, $\delta>0$ and $s \in [t,t+\delta]$. We use the triangle inequality and verify that:
$$
\|B_s^{(n)}-B_t^{(n)}\| \leq \sum_{k=N(n,t)}^{N(n,t+\delta)+1} \|H_{k-1}(\mathbb{Z}_{k-1})\| \lesssim \left\|
\sum_{k=N(n,t)}^{N(n,t+\delta)+1}   \frac{1}{k} \mathbb{Z}_{k-1}\right\|+  \left\|\sum_{N(n,t)}^{N(n,t+\delta)+1} \tilde{\mathcal{E}}_k\right\|.
$$
For the first term of the right hand side, we use the Tchebychev inequality and obtain that:
\begin{equation*}
\mathbb{P}\left( \left| \sum_{N(n,t)}^{N(n,t+\delta)+1}   \frac{1}{k} \mathbb{Z}_{k-1} \right| \geq \epsilon \right) \leq \epsilon^{-2} \mathbb{E} \left( \Bigl(\sum_{N(n,t)}^{N(n,t+\delta)+1}   \frac{\|Z_{k-1}\|}{k}\Bigr)^2 \right).
\end{equation*}
Now, the Cauchy-Schwarz inequality implies that:
$$
\left( \sum_{N(n,t)}^{N(n,t+\delta)+1}   \frac{\|Z_{k-1}\|}{k} \right)^2 \leq \left( \sum_{N(n,t)}^{N(n,t+\delta)+1} \frac{1}{k} \right) 
\left( \sum_{N(n,t)}^{N(n,t+\delta)+1}
   \frac{\|Z_{k-1}\|^2}{k} \right).
$$
Since $\sup_{k \geq 1} \mathbb{E} \|\mathbb{Z}_k\|^2 < + \infty$ and $\sum_{N(n,t)}^{N(n,t+\delta)+1}   \frac{1}{k}  \leq 2 \delta$ for large $n$, we then obtain that:
\begin{equation}\label{eq:tension1}
\mathbb{P}\left( \left| \sum_{N(n,t)}^{N(n,t+\delta)+1}   \frac{1}{k} \mathbb{Z}_{k-1} \right| \geq \epsilon \right) \lesssim
\epsilon^{-2} \delta^2.
\end{equation}

We handle the rest term in the same way and control $\left\|\sum_{N(n,t)}^{N(n,t+\delta)+1}  \tilde{\mathcal{E}}_k \right\|^2$. We apply Equation \eqref{eq:reste}, the triangle inequality and  the Minkowski inequality and deduce that:
\begin{align*}
\left\|\sum_{N(n,t)}^{N(n,t+\delta)+1}  \tilde{\mathcal{E}}_k \right\|^2& \lesssim \left(  \sum_{N(n,t)}^{N(n,t+\delta)+1} 
\frac{|\bar{\theta}_k-\ta|}{k^{3/2}}+
\frac{|\bar{\theta}_k-\ta|^2}{\sqrt{k}} + 
\frac{|\theta_k-\ta|}{k^{3/2-a}}+ \right. \\
&\qquad \left.
\frac{|\theta_k-\ta|^2}{\sqrt{k}} + 
\frac{|\widehat{\vartheta}_k-\vta|}{k^{3/2}} 
  \right)^2 \\
  & \lesssim  \left(\sum_{N(n,t)}^{N(n,t+\delta)} 
\frac{|\bar{\theta}_k-\ta|}{k^{3/2}}\right)^2 + 
\left( \sum_{N(n,t)}^{N(n,t+\delta)} 
\frac{|\bar{\theta}_k-\ta|^2}{\sqrt{k}} \right)^2 
+ \qquad\left( \sum_{N(n,t)}^{N(n,t+\delta)} 
\frac{|\theta_k-\ta|}{k^{3/2-a}}\right)^2 \\
& \qquad+ \left( \sum_{N(n,t)}^{N(n,t+\delta)} 
\frac{|\theta_k-\ta|^2}{\sqrt{k}}\right)^2 +
 \left(\sum_{N(n,t)}^{N(n,t+\delta)} 
\frac{|\widehat{\vartheta}_k-\vta|}{k^{3/2}} 
  \right)^2.
  \end{align*}
  Now we apply the Cauchy Schwarz inequality in order to use $\left(\sum_{N(n,t)}^{N(n,t+\delta)+1} \frac{1}{k}\right)$ in each term:
  \begin{align*}
\left\|\sum_{N(n,t)}^{N(n,t+\delta)+1}  \tilde{\mathcal{E}}_k \right\|^2&\lesssim \left(\sum_{N(n,t)}^{N(n,t+\delta)+1} 
\frac{1}{k}\right)\\&\times \left[
\sum_{N(n,t)}^{N(n,t+\delta)+1} 
\frac{|\bar{\theta}_k-\ta|^2}{k^2} + 
|\bar{\theta}_k-\ta|^4   +\frac{|\theta_k-\ta|^2}{k^{2-2a}}+|\theta_k-\ta|^4+\frac{|\widehat{\vartheta}_k-\vta|^2}{k^2} \right].
\end{align*}
We shall now compute the expectation of the previous terms and verify that 
\begin{align*}
\E \left\|\sum_{N(n,t)}^{N(n,t+\delta)}  \tilde{\mathcal{E}}_k \right\|^2 &\lesssim 
\left(\sum_{N(n,t)}^{N(n,t+\delta)} 
\frac{1}{k} 
  \right) \left(\sum_{N(n,t)}^{N(n,t+\delta)} 
\frac{1}{k} \left[ \frac{\E k |\bar{\theta}_k-\ta|^2}{k^2}
+ \frac{\E k^2 |\bar{\theta}_k-\ta|^4}{k} \right. \right.
\\
& \left. \left. 
+ \frac{k^a \E|\theta_k-\ta|^2}{k^{1-a}} + \frac{\E k^{2a} |\theta_k-\ta|^4}{k^{2a-1}}+\frac{\E k |\widehat{\vartheta}_k-\vta|^2}{k^2}
\right]
  \right).
\end{align*}
Using now Theorem \ref{theo:momentp} ($L^p$ loss on the sequence $(\tn)_{n \geq 1}$), Theorem \ref{theo:momentp} ($L^p$ loss on the sequence $(\btn)_{n \geq 1}$) and Theorem \ref{theo:rate_vtn} ($L^2$ loss on the sequence $(\vtn)_{n \geq 1}$),
we obtain that the bracket term is uniformly bounded in $k$, and therefore:
$$
\E \left\|\sum_{N(n,t)}^{N(n,t+\delta)}  \tilde{\mathcal{E}}_k \right\|^2 \lesssim \left(\sum_{N(n,t)}^{N(n,t+\delta)} 
\frac{1}{k} 
  \right)^2 \lesssim \delta^2.
$$
We then apply the Tchebychev inequality and obtain that:
\begin{equation}\label{eq:tension2}
\mathbb{P}\left( \left| \sum_{N(n,t)}^{N(n,t+\delta)}  \tilde{\mathcal{E}}_k \right| \geq \epsilon \right) \lesssim
\epsilon^{-2} \delta^2.
\end{equation}
Hence, for any $\epsilon>0$, for any $\eta>0$, a small enough $\delta$ exists (of the order $\eta \epsilon^2$)  and a large enough $n_0$ exists such that:
$$
\forall n \geq n_0 \qquad 
\mathbb{P}\left( \left| \sum_{N(n,t)}^{N(n,t+\delta)}  \tilde{\mathcal{E}}_k \right| \geq \epsilon \right) \leq \eta \delta.
$$
Theorem 7.3 of \cite{billingsley} implies that $(B^{(n)})_{n \geq 1}$ is a tight sequence of continuous processes.

\noindent
\underline{Step 3: Study of $M_s^{(n)}-M_t^{(n)}$ - Tightness of $(M^{(n)})_{n \geq 1}$.}
Since $s \in [t,t+\delta]$, $M_s^{(n)}$ safisfies:
$$
\|M_s^{(n)}-M_t^{(n)}\| \leq \max(\|M^{(n)}_{N(n,s)}-M_t^{(n)}\|,\|M^{(n)}_{N(n,s)+1}-M_t^{(n)}\|).
$$
Considering  now the supremum over $s \in [t,t+\delta]$, we then deduce that:
$$
\left\{ \sup_{s \in [t,t+\delta]} \|M_s^{(n)}-M_t^{(n)}\| \geq \epsilon \right\} 
\subset  \left\{ \sup_{N(n,t)+1 \leq k \leq N(n,t+\delta)+1} \|M_{\Gamma_k}^{(n)}-M_t^{(n)}\| \geq \epsilon\right\}.
$$
We then deduce that for any $p>2$:
\begin{align*}
\mathbb{P}\left( \sup_{s \in [t,t+\delta]} \|M_s^{(n)}-M_t^{(n)}\| \geq \epsilon \right)& \leq \mathbb{P}\left(\sup_{N(n,t)+1 \leq k \leq N(n,t+\delta)+1} \|M_{\Gamma_k}^{(n)}-M_t^{(n)}\| \geq \epsilon \right) \\
& \leq \epsilon^{-p} \E \left[ \sup_{N(n,t)+1 \leq k \leq N(n,t+\delta)+1} \|M_{\Gamma_k}^{(n)}-M_t^{(n)}\|^p\right] \\
& \lesssim \left(\frac{p}{p-1}\right)^p \epsilon^{-p} \E \left[\|M_{\Gamma_{N(n,t+\delta)+1}}^{(n)}-M_t^{(n)}\|^p\right],
\end{align*}
where we first apply the Markov inequality and second the Doob maximal inequality.
Now, the Minkowski inequality yields:
\begin{align*}
\mathbb{P}\left( \sup_{s \in [t,t+\delta]} \|M_s^{(n)}-M_t^{(n)}\| \geq \epsilon \right) &\leq \left(\frac{p}{p-1}\right)^p \epsilon^{-p}
\sum_{k=N(n,t)}^{N(n,t+\delta)+1} k^{-p/2} \E  \left\| \left(\begin{matrix}
k a_{k-1} (\frac{1}{k}-\epsilon_k) \Delta M_k \\ \sqrt{b_1} (1-\alpha) \sqrt{\frac{k+1}{k}} \Delta N_k
\end{matrix}\right)\right\|^p
\end{align*}
First, the martingale increments $ (\Delta M_n)_{n \geq 1}
$ are uniformly bounded (see Equation \eqref{eq:mart_inc}). Second, from our assumption $\mathbf{H}_f$,  the random variable $X$ possesses a moment of order $p>2$ and for this value of $p$, we observe that:
\begin{align*}
\mathbb{P}\left( \sup_{s \in [t,t+\delta]} \|M_s^{(n)}-M_t^{(n)}\| \geq \epsilon \right) & \lesssim \epsilon^{-p}\sum_{k=N(n,t)}^{N(n,t+\delta)+1} k^{-p/2}.
\end{align*}
Hence, we obtain that:
$$
\mathbb{P}\left( \sup_{s \in [t,t+\delta]} \|M_s^{(n)}-M_t^{(n)}\| \geq \epsilon \right) \leq \epsilon^{-p} N(n,t)^{-(p-2)/2} \sum_{k=N(n,t)}^{N(n,t+\delta)+1} \frac{1}{k}\leq \frac{\delta }{\epsilon^p N(n,t)^{(p-2)/2}}.
$$
We conclude that for this value of $p>2$, for any $\epsilon>0$ and any $\eta>0$, a $n_0$ exists such that:
$$
\forall n \geq n_0 \qquad 
{\Gamma_n}^{(p-2)/2} \geq \frac{1}{\epsilon^p \eta},
$$
which in turn implies that:
$$
\mathbb{P}\left( \sup_{s \in [t,t+\delta]} \|M_s^{(n)}-M_t^{(n)}\| \geq \epsilon \right) \leq \eta \delta,
$$
because $N(n,t) \geq N(n,0) = \Gamma_n$.
Again, Theorem 7.3 of \cite{billingsley} implies that $(M^{(n)})_{n \geq 1}$ is a tight sequence of continuous stochastic processes. It concludes the proof of the $i)$.

\noindent
\underline{Proof of $ii)$:}
We start from the definition of $\Delta \mathcal{M}_{n+1}$ given in Proposition \ref{prop:lin_TCL}:
\begin{align*}\E[\Delta&\mathcal{M}_{n+1}\Delta\mathcal{M}_{n+1}^t \lvert \mathcal{F}_n] = 
\begin{pmatrix}
S^{1,1}_n & S^{1,2}_n \\ 
S^{2,1}_n & S^{2,2}_n
\end{pmatrix},
\end{align*}
where 
$$
S^{1,1}_n = 
a_n^2(n+1)^2(\frac{1}{n+1}-\epsilon_{n+1})^2\E(\{\Delta M_{n+1}\}^2 \vert \Fn) \quad \text{and} \quad 
S^{2,2}_n = \frac{b_1(n+2)}{(1-\alpha)^2(n+1)}\E(\{\Delta N_{n+1}\}^2\vert \Fn),
$$
and
$$
S^{1,2}_n = S^{2,1}_n = a_n(n+1)\left(\frac{1}{n+1}-\epsilon_{n+1}\right)\frac{\sqrt{b_1}}{(1-\alpha)}\sqrt{\frac{n+2}{n+1}}\E(\Delta M_{n+1}\Delta N_{n+1}\vert \Fn).
$$
We now study each term separately.\\
$\bullet$ Lemma \ref{lem:t1} implies that:
\begin{align*}S_n^{1,1} %&= a_{n}^2 (n+1)^2\left( \epsilon_{n+1}-\frac{1}{n+1}\right)^2 \E \left[ \{\Delta M_{n+1}\}^2\vert \Fn\right]\\
&  =\frac{\alpha(1-\alpha)}{ f(\ta)^2} + O \left( |\tn-\ta|\right).\end{align*}
$\bullet$ Proposition \ref{prop:linearization} yields:
$$S_n^{2,2}% = \frac{b_1(n+2)}{(1-\alpha)^2(n+1)}\E(\{\Delta N_{n+1}\}^2\lvert \mathcal{F}_{n})
=\frac{b_1 V_\alpha}{(1-\alpha)^2}+O(|\btn-\ta|).$$
$\bullet$ Finally, we deduce from Lemma \ref{lem:tec_termes_martingales} $iii)$ that
% $$\E(\Delta M_{n+1}\Delta N_{n+1})=\alpha(1-\alpha)\vta +O(n^{-a/4}).$$
%Combining with $a_n(n+1)(\frac{1}{n+1}-\epsilon_{n+1})=\frac{1}{f(\ta)}+O(n^{a-1})$ we deduce that 
\begin{align*}
S_n^{1,2}& = 
a_n(n+1)(\frac{1}{n+1}-\epsilon_{n+1})\frac{\sqrt{b_1}}{(1-\alpha)}\sqrt{\frac{n+2}{n+1}}\E(\Delta M_{n+1}\Delta N_{n+1}\vert \Fn)\\
&=\frac{\sqrt{b_1}\alpha\vta}{f(\ta)}+O\left(\sqrt{|\btn-\ta|}+\sqrt{|\tn-\ta|}+|\btn-\ta|+|\tn-\ta|\right).\end{align*} 
Using our previous results on $(\tn)_{n \geq 1}$ and $(\btn)_{n \geq 1}$,  the conclusion holds.

\noindent
\underline{Proof of $iii)$:}
We consider a twice differentiable test function $\varphi$. Using a second order Taylor expansion and Proposition \ref{prop:linearization}, we obtain that:
$$
\E \left[\varphi(\Tznp) \, \vert \Tzn \right] = \varphi (\Tzn) + \frac{1}{n+1} \mathcal{G}(\varphi)(\Tzn)+\mathbb{Q}_n,
$$
where $\mathcal{G}(\varphi)$ is defined in Equation \eqref{eq:generateur}.
Finally, combining \eqref{eq:reste}, Theorem\ref{theo:rate_averaging}, Theorem \ref{theo:rate_vtn} and the fact that the martingale increments are bounded, we deduce that 
$$E((n+1)\lvert \mathbb{Q}_n\lvert)\underset{n\to\infty}{\longrightarrow}0.$$

\end{proof}

\end{document}